\documentclass[11pt]{amsart}
\usepackage[pdftex]{graphicx} 
\usepackage{color}
\usepackage[T1]{fontenc}
\usepackage{lmodern}
\usepackage[english]{babel} 
\usepackage[utf8]{inputenc}

\usepackage[centering,margin=2.5cm]{geometry}
\usepackage{upgreek}

\usepackage{amsfonts}
\usepackage{verbatim}
\usepackage{amsmath}
\usepackage[all, cmtip]{xy}\usepackage{amssymb}
\usepackage{amsthm}
\usepackage{extpfeil}
\usepackage{bbm}
\usepackage{tikz-cd}
\usepackage{paralist}
\usepackage[T1]{fontenc}
\usepackage[colorlinks=false,citecolor=magenta,pagebackref=true,urlcolor=magenta,pdftex]{hyperref}

\usepackage{amscd}

\usepackage{nicefrac}
\usepackage{microtype}
\usepackage{epstopdf}

\usepackage{mathrsfs}
\usepackage[font=small,labelfont=bf]{caption}
\usepackage[labelformat=simple]{subcaption}

\usepackage[nodisplayskipstretch]{setspace}
\makeatletter
\newtheorem*{rep@theorem}{\rep@title}
\newcommand{\newreptheorem}[2]{%
\newenvironment{rep#1}[1]{%
 \def\rep@title{#2~\ref{##1}}%
 \begin{rep@theorem}}%
 {\end{rep@theorem}}}
 
\makeatother

\theoremstyle{plain}
\newtheorem*{thm*}{Theorem}
\newtheorem{thm}{Theorem}[section]
\newtheorem{mthm}[thm]{\bf Theorem}

\newtheorem{mcor}[thm]{\bf Corollary}
\newtheorem{mlem}[thm]{\bf Lemma}

\newreptheorem{mthm}{Theorem}
\newreptheorem{mcor}{Corollary}
\newreptheorem{mlem}{Lemma}
\newreptheorem{thm}{Theorem}

\newtheorem{cor}[thm]{Corollary}
\newtheorem{lem}[thm]{Lemma}
\newtheorem*{lem*}{Lemma}
\newtheorem{prp}[thm]{Proposition}

\newtheorem{prb}[thm]{Open Problem}

\theoremstyle{definition}
\newtheorem{dfn}[thm]{Definition}
\newtheorem{ex}[thm]{Example}

\newtheorem{rem}[thm]{Remark}

\newtheorem{rmk}[thm]{Remark}

\newcommand{\supp}{\mathrm{supp}}

\newcommand{\defeq }{\stackrel{\rm def}{=}}

\newcommand{\wh}{\widehat}
\newcommand{\wt}{\widetilde}

 \newcommand{\Z}{\mathbb{Z}}
 \newcommand{\R}{\mathbb{R}}

\newcommand{\BB}{\mathcal{B}}

\newcommand{\FF}{\mathcal{F}}

\newcommand{\LL}{\mathcal{L}}

\newcommand{\rarr}{\rightarrow}

\newcommand{\Lrarr}{\Longleftrightarrow}

\newcommand{\vanish}[1]{}

\newcommand{\De}{\Delta} 
\newcommand{\Ga}{\Gamma} 
 \newcommand{\la}{\lambda}

\def\om{\omega}

\def\si{\sigma}

\def\sbs{\subset}

\def\wh{\widehat}
\def\wt{\widetilde}

\def\langle{\left<}
\def\rangle{\right>}

\def\({\left(}
\def\){\right)}
\def\no={\,{\,|\!\!\!\!\!=\,\,}}
\def\wt{\widetilde}

\def\wh{\widehat}
\def\no={\,{\,|\!\!\!\!\!=\,\,}}

\def\rk{\mathrm{rk}}

\def\sbs{\subset}

\newcommand{\xqedhere}[2]{%
 \rlap{\hbox to#1{\hfil\llap{\ensuremath{#2}}}}}

\newcommand{\cm}[1]{}
\newcommand\mc[1]{\mathcal{#1}}
\newcommand\mbb[1]{\mathbb{#1}}
\newcommand\mbf[1]{\mathbf{#1}}
\newcommand\frk[1]{\mathfrak{#1}}
\newcommand\mr[1]{\mathrm{#1}}
\newcommand\ol[1]{\overline{#1}}

\newcommand{\longhookrightarrow}{\ensuremath{\lhook\joinrel\relbar\joinrel\longrightarrow}}

\newcommand{\bigslant}[2]{{\raisebox{.3em}{$#1$} \Big/ \raisebox{-.3em}{$#2$}}}
\renewcommand\emptyset{\varnothing} 

\newcommand\TR{\mathbb{T}}
\newcommand\TP{\mathbb{TP}}
\newcommand\tp{\, \oplus\, }
\newcommand\tm{\, \odot\, }

\newcommand\B{\frk{B}\, }
\newcommand\Fan{\frk{F}\, }

\newcommand\sed{\frk{s}\, }
\newcommand\Sed{\frk{S}\, }
\newcommand\mo{\frk{m}\, }

\newcommand{\NS}{\mc{NS}}
\newcommand\RR{\mathbb{R}}
\newcommand\NA{\mathbb{N}}

\DeclareMathOperator\TT{T}
\DeclareMathOperator\RN{N}

\newcommand\CO{\mc{C}\,}
\DeclareMathOperator\lin{lin}
\DeclareMathOperator\aff{aff}
\DeclareMathOperator\pos{pos}

\DeclareMathOperator{\Lk}{lk}
\DeclareMathOperator{\St}{st}

\DeclareMathOperator{\tT}{tT}

\newcommand{\RS}[2]{#1_{|#2}}
\newcommand{\rint}[1]{{#1}^\circ}

\title[Tropical Lefschetz Section Theorems]{Filtered geometric lattices and\\ Lefschetz Section Theorems over the tropical semiring}
\author{Karim~Adiprasito}
\author{Anders Bj\"orner}
\address{
Einstein Institute of Mathematics, Hebrew University of Jerusalem, Jerusalem, Israel}
\email{adiprasito@mail.huji.ac.il}
\address{Institut Mittag-Leffler, Djursholm,
Sweden 
\newline
and
\newline
Department of Mathematics, %
Kungliga Tekniska H\"ogskolan, Stockholm, %
Sweden}
\email{bjorner@kth.se}

\date{\today}
\thanks{K.~Adiprasito acknowledges support by ISF Grant 1050/16 and ERC StG 716424 - CASe.
}
\thanks{A. Bj\"orner acknowledges support from the Swedish National Research Council (VR)
under Grant No. 2015-05308, from the Fondazione 
San Michele, and from the National 
Science Foundation under Grant No. DMS-1440140 while the author was in residence
 at MSRI, Berkeley, November 2013, and November 2017.}

\begin{document}
\begin{abstract}
The purpose of this paper is to establish analogues of the classical Lefschetz Section Theorem for smooth tropical varieties. 
We attempt to give a comprehensive picture of the Lefschetz Section Theorem in tropical geometry by deriving tropical analogues of many of the classical Lefschetz theorems and associated vanishing theorems. 

We start the paper by resolving a conjecture of Mikhalkin and Ziegler concerning the homotopy type of certain filtrations of geometric lattices, generalizing several known properties of full geometric lattices. This translates to a crucial index estimate for the stratified Morse data at critical points of the tropical variety. Finally, we extend this to prove a Lefschetz theorem for the Itenberg-Mikhalkin-Katzarkov-Zharkov tropical Hodge theory. The tropical varieties that 
we deal with are locally matroidal, and the Lefschetz theorems are shown to apply also in the case of nonrealizable matroids.



\end{abstract}
\maketitle

Tropical geometry is a relatively young field of mathematics, based on early work of, among many others, Bergman \cite{Bergman} and Bieri--Groves \cite{BieriGroves}. It arises as algebraic geometry over the tropical max-plus semiring $\TR=([-\infty,\infty),\max,+)$, and also as a limit of classical complex algebraic geometry. Since tropical varieties are, in essence, polyhedral spaces, tropical geometry naturally connects the fields of algebraic and combinatorial geometry, and combinatorial tools have proven essential for the study of tropical varieties.

Since its origins, tropical geometry has been extensively developed \cite{Gathmann, RGST, Speyer, SpeyerSturmfels}. It has been applied to classical algebraic geometry \cite{Gubler}, enumerative algebraic geometry \cite{KT, MG, MG2, Shustin}, mirror symmetry \cite{Gross, KoSo}, integrable systems \cite{TGaIS}, and to several branches of applied mathematics, cf.\ \cite{MP, TGI}. Several classical results in algebraic geometry have natural analogues in tropical geometry, see e.g.\ \cite{CDPR}.

The purpose of this paper is to provide tropical analogues and extensions of one of the most central results in algebraic geometry, the Lefschetz Section Theorem (or Lefschetz Hyperplane Theorem). This is motivated not so much by a desire to reprove classical algebrao-geometric theorems in a tropical setting, as
by the recent surge of interest in the algebraic geometry and topology of tropical and combinatorial varieties in the context of combinatorial Hodge theory \cite{AHK} 
and nonarchimedean spaces \cite{CLD, HrL}. This in turn is primarily fueled by the discovery that certain combinatorial structures satisfy laws of the Lefschetz and Hodge type, even though this is not predicted by the classical algebraic formalism. The results of the present paper provide prime examples of this phenomenon. 

Our aim is to present a comprehensive picture of the Lefschetz Section Theorem in tropical geometry, and to give tropical analogues of many of the classical Lefschetz theorems and associated vanishing theorems. It should be noticed,
however, that the results we obtain are so general that in some
cases it is a bit tricky to directly compare them to the classical setting, which will then appear as a distant inspiration.
Along the way, we build on and generalize significant results in the topological theory of posets and matroids of Rota, Folkman, Quillen
and others.

{\setcounter{tocdepth}{1}
\tableofcontents
\section*{The Tropical Lefschetz Section Theorems.} The classical Lefschetz Section Theorem comes in many different guises. Intuitively, Lefschetz theorems relate the topology of a complex algebraic variety $X$ to that of the intersection of $X$ with a hyperplane $H$ transversal to $X$ (or, in abstract settings, to an ample divisor $D$ of $X$). In its most classical form, it relates the homology groups of $X$ and $X\cap H$ \cite{Lefschetz, AndreottiFrankel} for smooth complex projective algebraic varieties. 

Since Lefschetz' pioneering work, many different variants of this important theorem were established. There are versions 
for affine varieties and projective varieties, for homology groups and homotopy groups, for Hodge groups and Picard groups, for constructible sheaves, and several more; compare \cite{GM-SMT, Lazarsfeld, Voisin}. In their dual formulations, Lefschetz theorems turn into vanishing theorems, such as the Andreotti--Frankel \cite{AndreottiFrankel}, Akizuki--Kodaira--Nakano \cite{AN} and Grothendieck--Artin \cite{Lazarsfeld} Vanishing Theorems.

In this paper we establish analogues in tropical geometry of several of the classical Lefschetz theorems. More precisely, we shall provide tropical analogues of
\begin{compactitem}[$\bullet$]
\item the Andreotti--Frankel Vanishing Theorem for affine varieties.
\item the classical Lefschetz Section Theorem for homology groups of projective varieties, due to Lefschetz and Andreotti--Frankel \cite{AndreottiFrankel, Lefschetz}.
\item the Bott--Milnor--Thom Lefschetz Section Theorem for homotopy groups and CW models of projective varieties \cite{Bott, Milnor}.
\item the Hamm--L{\^e} Lefschetz Section Theorem for complements of affine varieties \cite{HammLe}.
\item the Akizuki--Kodaira--Nakano Lefschetz Vanishing Theorem for Hodge groups \cite{AN, Voisin}.
\item the Kodaira--Spencer Lefschetz Section Theorem for Hodge groups \cite{KS}.
\end{compactitem}
An interesting feature of our theorems is that they apply even to situations that are not attainable as limits of the classical algebraic situation, to matroids that are not realizable over any field.
Compare this also to the Hard Lefschetz Theorem for matroids proven in \cite{AHK}.

\subsection*{Tropical Lefschetz Section Theorems for CW models, homotopy and homology.} The first section theorem of this paper is an analogue of the Andreotti--Frankel Vanishing Theorem \cite{AndreottiFrankel} for smooth tropical varieties. 
The \emph{tropical affine space} $\TR^d$ of dimension $d$ is the space $[-\infty,\infty)^d$. 
We use $\mathbb{A}^d$ to denote any one of $\TR^d$, $\TP^d$ or $\mathbb{R}^d$. In fact, our theorems extend further to closed polyhedra in $\mathbb{A}^d$, and in particular to the case of $\mathbb{A}^d$ a projective tropical toric variety, see \cite[Section 6]{ms}.

\begin{repmthm}{mthm:lef_t_trop_aff}
Let $X\subset \mathbb{A}^d$ be a smooth $n$-dimensional tropical variety, and let $H$ denote a chamber complex in $\mathbb{A}^d$. Then $X$ is, up to homotopy equivalence, obtained from $X\cap H$ by successively attaching $n$-dimensional cells. 

In particular, the inclusion $X\cap H \hookrightarrow X$ induces isomorphisms of homotopy groups resp.\ integral homology groups up to dimension $n-2$, and a surjection in dimension~$n-1$.
\end{repmthm}

Here, a {chamber complex} is a polyhedral complex that divides its ambient space into pointed polyhedra, that is, polyhedra that do not contain lines; the notion generalizes tropical hypersurfaces and hyperplanes (See Section~\ref{sec:pT}) for a complete definition). Contrary to the original treatment of Andreotti--Frankel, this result does not follow immediately from Lefschetz duality and the affine theorem, but rather from a common generalization of the affine and projective cases (Lemma~\ref{mlem:lef_to_convex_cell}).

\begin{rem}[Stable intersection] Theorem~\ref{mthm:lef_t_trop_aff}, as well as Theorem~\ref{mthm:proj_t_hodge_g} below, hold more generally also for \emph{stable intersections} of $X$ with a chamber complex, that is, limits of intersections $H+tv$ as $t$ tends to $0$, where $H$ is a chamber complex as above, $v$ is any generic vector, \emph{as long as this is well defined independently of $v$}. We denote this by
\[X\cap_{\mathrm{stable}} H \ \defeq \ \lim_{t\rightarrow 0}\ X\cap H+tv.\]
This is the case, for instance, if $H$ is the support of a tropical hypersurface, that is, admits a positive weight balancing with respect to some triangulation \cite{ms}. We shall indicate the proof of these versions when we come to the proof of the respective theorems, see Section~\ref{sec:stable}.
\end{rem}

\subsection*{Tropical Lefschetz Section Theorems for complements of tropical varieties.} Our reasoning extends to the complement of a tropical variety as well. This is analogous to the Hamm--L{\^e} Lefschetz theorems \cite{HammLe} for complements of algebraic hypersurfaces. 

\begin{repmthm}{mthm:left_comp_aff}
Let $X$ denote a smooth $n$-dimensional tropical variety in $\mathbb{A}^d$, and let $C=C(X)$ denote the complement of $X$ in $\mathbb{A}^d$. Let 
furthermore $H$ denote the closure of a real hyperplane in~$\mathbb{A}^d$
such that for any face $\sigma \in X$, $\aff \sigma$ intersects $H$ transversally or not at all. Then $C$ is, up to homotopy equivalence, obtained from $C\cap H$ by successively attaching $(d-n-1)$-dimensional cells. 

\end{repmthm}


The main tool to prove this result is the construction of an efficient Salvetti-type complex for complements of Bergman fans. In this connection we also   characterize homotopically the ``complement'' of a matroid,
see Corollary~\ref{mcor:Comp_Bergman_Fan}.

\subsection*{Tropical Lefschetz Section Theorems for  $(p,q)$-groups.} Finally, we provide a Lefschetz theorem for $(p,q)$-homology,
a concept introduced by Itenberg--Katzarov--Mikhalkin--Zharkov~\cite{IKMZ}, which can be seen as a tropical Hodge theory.

This is nontrivial: While the classical Lefschetz Section Theorem for Hodge groups of smooth algebraic projective varieties (due to Kodaira--Spencer \cite{KS}) does follow from the Lefschetz Section Theorem for complex coefficients and the Hodge Decomposition, this approach does not apply here. Nevertheless, the Lefschetz Section Theorem holds true for $(p,q)$-groups. 


\begin{repmthm}{mthm:proj_t_hodge_g}
Let $X$ denote an $n$-dimensional smooth tropical variety in $\mathbb{A}^d$, and let $H\subset \mathbb{A}^d$ denote an ample {{chamber complex}}. Then the inclusion $X\cap H\hookrightarrow X$ induces an isomorphism of $(p,q)$-homology
\[H_q(X\cap H; \FF_p (X \cap H))\ \longrightarrow\ H_q(X; \FF_p X) \]
for $p+q\le n-2$, and a surjection 
when $p+q=n-1$.
If $H$ has positive balancing (so that the stable intersection is well-defined), then we moreover have
\[H_q(X\cap_{\mathrm{stable}} H; \FF_p (X \cap_{\mathrm{stable}} H))\ \longrightarrow\ H_q(X; \FF_p X).\]
\end{repmthm}

We refer to Section~\ref{sec:pT} for the definition of ampleness in the tropical context. 

For the proof, instead of invoking a tropical Hodge Decomposition Theorem and Theorem~\ref{mthm:lef_t_trop_aff}, we establish a tropical analogue of the Akizuki--Kodaira--Nakano Vanishing Theorem, 
see Theorem~\ref{mthm:AKN_trop}.
The proof of the tropical Kodaira--Spencer Theorem can then be finished in a manner similar to the classical proof via the long exact sequence of Hodge groups, compare also \cite{AN, Voisin}. 

We also provide a counterexample to the integral version of the above theorem (see Section~\ref{ssc:intHodge}).

\section*{Filtered geometric lattices.} For the proofs of the tropical Lefschetz theorems, we shall critically use stratified Morse theory (\cite{GM-SMT}, see also our brief introduction to stratified Morse theory in Section~\ref{sec:polyhedra}). A crucial ingredient of the Morse-theoretic approach to classical Lefschetz Theorems are estimates on Morse indices at critical points, which follow easily from general considerations on Hessians of homogeneous complex polynomials, cf.\ \cite{AndreottiFrankel, Lazarsfeld, Milnor}. 

In our setting, we analogously estimate the topological changes in the sublevel sets with respect to some smooth Morse function, interpreting the tropical variety as a Whitney stratified space. This requires
of  us to verify a conjecture of Mikhalkin and Ziegler \cite{MZ} 
concerning topological properties of {geometric lattices}. The terminology and notation used here is explained in Section~\ref{sec:pos_sum}.

\begin{repmthm}{mthm:pos_sum_geom_latt_cm}
Let $\LL$ denote the lattice of flats of a matroid on ground set $[n]$ and of rank $r\ge 2$, and let $\omega$ denote a generic weight on its atoms. Let $t$ denote any real number with $t\le \min\{0,\omega\cdot [n]\}$. Then $\LL^{>t}$ is homotopy Cohen--Macaulay of dimension $r-2$. In particular, it is $(r-3)$-connected.
\end{repmthm}

For the proof, we rely on lexicographic shellability and a generalization of Quillen's fiber lemma.

Of geometric interest is that this result implies a topological characterization of half-links of a smooth tropical variety at critical points. 

Of combinatorial interest is that Theorem~\ref{mthm:pos_sum_geom_latt_cm} gives deeper information about the structure of geometric lattices,
generalizing earlier characterizations of the topological type of 
the full geometric lattice $\LL$, the homology version of which goes back to work of Folkman
\cite{Folkman}, inspired by work
of Rota \cite{Rota} on the M\"obius function of geometric lattices.
A stronger version concerning shellability, and therefore also homotopy equivalence,
was later proved by Bj\"orner \cite{Bj1}. 
In Section \ref{ssc:crime}, we illustrate the combinatorial content of  Theorem~\ref{mthm:pos_sum_geom_latt_cm} 
by a mock application to graph-connectivity and crime-prevention.

Other than for full geometric lattices, Theorem~\ref{mthm:pos_sum_geom_latt_cm} was previously known for Boolean lattices \cite{Cortona} (equivalently: free matroids), for lattices of rank $3$ \cite{PZ}, and also for the case when the weight $\omega$
has only one negative entry (this is implied by a result of
Wachs and Walker \cite{WachsWalker}).


\section*{Plan for the paper.}

In Section~\ref{sec:pos_sum}
we prove our main theorem on the homotopy Cohen--Macaulayness of filtered geometric lattices, using methods from poset topology. Also, in Remark 
3.2 we sketch an alternative proof based on the combinatorial Morse Theory of Forman \cite{FormanADV}.

In Sections~\ref{sec:cell},~\ref{sec:compl} and~\ref{sec:hodge}
we apply our results to deriving Lefschetz Theorems for tropical varieties. In each section, we first review a classical Lefschetz theorem, then proceed to give a tropical analogue.

In the remaining sections we review and extend required background information,
on combinatorial, cellular and poset topology, geometry and combinatorics of polyhedral complexes, tropical geometry, $(p,q)$-homology theory, and stratified Morse theory. While the main purpose for this material
is to provide a carefully laid foundation for the proofs of the main results, it is also meant to provide a coherent presentation of the material.

\noindent \emph{Acknowledgement.} We are grateful to G. Mikhalkin and G. M. Ziegler for communicating
their conjecture and for pointing out its relevance for tropical geometry, and to 
I.\ Zharkov for valuable comments. We also thank Johannes Rau for reminding us to add the case of stable intersections.

\part{Filtered Geometric lattices}

\label{sec:app} 
\section{Some basic combinatorial topology.}\label{ssc:ct} We recall some basic facts from algebraic topology and the topology of posets. The reader is referred to \cite{munkres}, \cite{Hatcher}, \cite{BTM} and \cite{Whitehead} for more details. All topological spaces have the homotopy type of simplicial complexes and, in particular, always have a CW decomposition.

\subsection{Polyhedral spaces and complexes.}\label{sec:11} A \emph{(closed) polyhedral complex} in $\RR^d$ is a finite collection of polyhedra in $\RR^d$ such that the intersection of any two polyhedra is a face of both, and that is closed under passing to faces of the polyhedra in the collection. The elements of a polyhedral complex are called \emph{faces}, and the inclusion-wise maximal faces are the \emph{facets} of the polyhedral complex. A polyhedral complex is \emph{bounded} if and only if all polyhedra are bounded, i.e.,\ if they are polytopes. Finally, a \emph{polyhedral fan} is a polyhedral complex all whose faces are polyhedra pointed at $\mbf{0}$, that is, their only vertex is the vertex at $\mbf{0}$. A \emph{pointed polyhedron} in general is a polyhedron that does not contain a line.
 We denote by
$\pos X$ the positive span of a subset $X$ of $\RR^d$, and $\lin X$ resp.\ $\aff X$ its linear and affine span, respectively.

The \emph{underlying (polyhedral) space} $|X|$ of a polyhedral complex $X$ is the union of its faces. With abuse of notation, we often speak of the polyhedral complex when we actually mean its underlying space;
for instance, for a subcomplex $Y \sbs X$
we write  simply $X \setminus Y$ for $| X | \setminus | Y |$.

If $\Ga^0$ resp.\ $\De^0$ denotes the vertex set of simplicial 
complexes $\Ga$ resp.\ $\De$, then we write $\De - \Ga$ for the subcomplex of $\De$ induced on
the vertex set $\De^0 \setminus \Ga^0$, and call this operation the \emph{deletion} of $\Ga$ from $\De$.

For polyhedral complexes, we define the \emph{combinatorial restriction} $\RS{X}{M}$ of a polyhedral complex $X$ to a set $M$ to be the inclusion-wise maximal subcomplex $Y$ of $X$ such that $Y\subset M$. We write $X-M$ for the subcomplex of $X$ given by $\RS{X}{X{{{\setminus}}} M}$ and call this operation the \emph{deletion} of $M$ from $X$.

If $X$ and $Y$ are two polyhedral complexes with the same underlying space, then $Y$ is called a \emph{refinement} or \emph{subdivision} of $X$ if every face $y$ of $Y$ is contained in some face $x$ of $X$. Similarly, for polyhedral complexes $X$, $Y$ we define the \emph{common refinement} $X\cdot Y$ as the polyhedral complex $\{x\hspace{0.08 em}\cap \hspace{0.08 em} y: x\in X,\ y\in Y\}$, so that in particular $|X\cdot Y|=|X|\cap |Y|$.

\subsection{Acyclicity, Connectivity and Cohen-Macaulayness.} A topological space $X$ is said to be \emph{$k$-connected} if one of the following equivalent conditions holds:
\begin{compactitem}[$\bullet$]
\item $\pi_i(X)=0$ for all $i\le k$, i.e., every map of 
the sphere $S^i$, $i\le k$, into $X$ is null-homotopic,
\item $X$ is homotopy equivalent to a CW complex that, except for the basepoint, has no cells of dimension~$\le k$.
\end{compactitem}
Similarly, a pair of topological spaces $(X,Y)$ is \emph{$k$-connected} 
if $\pi_i(X,Y)=0$ for all $i\le k$.

A space $X$ is \emph{$k$-acyclic} if $\widetilde{H}_i(X; \mathbb{Z})=0$ for all $i\le k$,
and a pair of spaces $(X,Y)$ is \emph{$k$-acyclic} if $\widetilde{H}_i(X,Y; \mathbb{Z})=0$
for all $i\le k$. 
By elementary cellular homology, every {$k$-connected} space is $k$-acyclic. We will repeatedly make use of the fact that by the theorems of Whitehead and Hurewicz (see for instance \cite[Section~4]{Hatcher}), a {$k$-acyclic} space (or pair of spaces), $k\ge 1$, is {$k$-connected} if and only if it is $1$-connected.


A pure simplicial complex (that is, a simplicial complex all whose facets are of the same dimension) $\Delta$ of dimension $d-1$ is \emph{homotopy Cohen--Macaulay} if any of the following equivalent conditions holds 
\begin{compactitem}[$\bullet$]
\item for all faces $\sigma$ in $\Delta$, the link $\Lk_\sigma \Delta$ is $(d-\dim \sigma-3)$-connected.
\item for all faces $\sigma$ in $\Delta$, the link $\Lk_\sigma \Delta$ is homotopy equivalent to a wedge of $(d-\dim \sigma-2)$-dimensional spheres.
\end{compactitem}
Here the empty set is considered to be a $(-1)$-dimensional 
face, and $\Lk_\emptyset \Delta=\Delta$. 

A pure $(d-1)$-dimensional simplicial complex $\Delta$ is \emph{Cohen--Macaulay over} $\Z$
if $ \wt{H}_{i}(\Lk_\sigma \Delta; \Z)=0$
for all faces $\sigma\in\Delta$ and all $i\le d-\dim \sigma-3$.
Being Cohen--Macaulay over ${\mathit R}$ is similarly defined for other coefficient rings ${\mathit R}$.
See \cite{SCCA} for some of the algebraic ramifications of this concept.

\subsection{Elementary cellular topology.} 

We now recall three well-known results in combinatorial topology.
The first one concerns the relation of the set $\De\setminus\Ga$ 
and the complex $\De-\Ga$ for simplicial complexes $\De$ and $\Ga$,
as defined in Section~\ref{sec:11}.
 
\begin{lem}\label{lem:deletion_in_simplicial}
Let $\De$ denote a simplicial complex, and let $V$ be a subset of its vertex set. If $\Ga$ denotes the subcomplex of $\De$ 
induced on vertex set $V$, then the set $\De\setminus \Ga$ deformation retracts to the subcomplex $\Delta - \Ga$ induced on $\De^0\setminus V$.
 \qed
\end{lem}

We use $A \ast B$ to denote the join of two topological spaces 
$A$ and  $B$ (or CW complexes, or simplicial complexes), and $\CO X \defeq  \{\text{point}\}\ast X$ to denote the (abstract) cone over a topological space $X$.

\begin{lem}\label{lem:join}
Let $\De$, $\Ga$ denote two topological spaces that are $k$-connected and $\ell$-connected, respectively. Then $\De\ast \Ga$ is $(k+\ell+2)$-connected.
\end{lem}

\begin{proof}
Let us consider the spaces $\De'\defeq \De\ast \Ga{{\setminus}} \Ga$ and $\Ga'\defeq \De\ast \Ga{{\setminus}} \De$. Then we have homotopy equivalences $\De'\simeq \De$, $\Ga'\simeq \Ga$ and $\De'\cap \Ga'\simeq \De\times \Ga$. By considering the Mayer--Vietoris sequence for $\De'$ and $\Ga'$, together with the K\"unneth formula, we see that $\De\ast \Ga$ is $(k+\ell+2)$-acyclic. The claim follows from the Whitehead and Hurewicz Theorems.
\end{proof}

If $X$ is any topological space, and $Y\subset X$ is any subspace, then we say that $X$ \emph{is obtained from} $Y$ by attaching an $i$-cell if $X$ can, up to homotopy equivalence, be decomposed as the union 
\[\bigslant{Y\cup e}{\alpha(x)\sim x}\]
where $e$ is an $i$-cell and $\alpha$ is a map $\alpha:\partial e\rightarrow Y$.

\begin{lem}\label{lem:lk}
Let $\De$ denote a polytopal complex, and let $\sigma$ be any $\ell$-cell of $\De$. If $\Lk_\sigma \Delta$ is $k$-connected, then $\De$ is, up to homotopy equivalence, obtained from $\De-\sigma$ by successively attaching cells of dimension $\ge k+\ell+2$.
\end{lem}

\begin{proof} By a stellar subdivision at $\sigma$ and Lemma~\ref{lem:join}, it suffices to address the 
case $\ell=0$, i.e.,the case when $\sigma=v$ is a vertex. We can furthermore assume that $k\ge 0$, since the claim is trivial otherwise.

Let $K$ denote a CW complex homotopy equivalent to $\Lk_v \Delta$ and constructed so that it has no nontrivial cells of dimension $\le k$.
Let $f: K\rightarrow\Lk_v \Delta$ denote a map realizing the homotopy equivalence 
$K \simeq \Lk_v \Delta$, and let \[M_f=\bigslant{K\times [0,1]\cup\Lk_v \Delta}{(x,0)\sim f(x)}\] denote its mapping cylinder. Then $\De$ is homotopy equivalent to 
\[\bigslant{((\De-v)\cup M_f) \cup \CO (K)}{x\in \partial (\CO (K)) \sim (x,1)}\]
Now, if $c$ is any nontrivial cell of $\partial (\CO (K))$, then $\CO (c)$ is a disk in $\CO (K)$ of dimension $\ge k+2$ (since $c$ is a cell of dimension $\ge k+1$). Since all nontrivial cells are of this form, the claim follows.
\end{proof}

\subsection{Topology of posets.} 
Posets $\mc{P}$ are interpreted topologically via their \emph{order complex} $\Delta(\mc{P})$,
whose faces are the totally ordered subsets (chains) of $\mc{P}$.
Here, $\Delta(\cdot)$ is usually suppressed from the notation. 
For instance, for a $(d-1)$-dimensional homotopy Cohen--Macaulay poset $\mc{P}$
as above, we have $\mc{P} \simeq \bigvee {S}^{d-1}$.
 As a general reference for poset topology, see \cite{BTM}.

A well-known consequence of Lemma~\ref{lem:join} (see e.g.\ \cite{quillen, BTM}) is
that a poset is Cohen--Macaulay (resp.\ homotopy CM) if and only if its intervals of length $k$
are $(k-1)$-acyclic (resp.\ $(k-1)$-connected) for all $k$. 
 
For a poset $\mc{P}$ and two comparable elements $a, b \in \mc{P}$, 
we have the \emph{interval}
$\mc{P}_{[a,b]}\defeq  \{y\in\mc{P} :a\le y\le b\}$ (and similarly for open and half-open
intervals). We also have the \emph{lower} (resp.\ \emph{upper}) \emph{ideal} $\mc{P}_{\le x}\defeq \{y\in\mc{P} :y\le x\}$ (resp.\ $\mc{P}_{\ge x}\defeq \{y\in\mc{P} :y\ge x\}$) of an element $x\in \mc{P}$. Upper ideals are also called \emph{order filters} in the literature.

An order-preserving map $f: \mc{P} \rarr \mc{P} $ is called a \emph{closure operator} if 
$x\le f(x)=f^2 (x)$ for all $x\in \mc{P} $. 
A closure operator is a strong deformation retract, see e.g. \cite[p.~1852]{BTM} or the proof of Lemma~\ref{mlem:quillen} below.
A concrete example of a closure operator that plays a role in this paper is
the closure map of a matroid, sending an arbitrary set of points to 
the smallest closed set containing it. A homotopy inverse 
of this closure operator is
the identity map, sending a closed set to itself.

The following version of Quillen's ``Theorem A'',
slightly more general than what can be found in the literature, cf.\ \cite{quillen, NFH, BWW}, is a central tool for our line of reasoning. 

\begin{lem}\label{mlem:quillen}
Let $\mc{P}$, $\mc{Q}$ be two posets, and $\varphi:\mc{P}\rightarrow \mc{Q}$ an
order-preserving map. 
Assume that for every $x\in \mc{Q}$, the fiber $\varphi^{-1}(\mc{Q}_{\le x})$ is $m_x$-connected and the upper ideal $\mc{Q}_{> x}$ is 
$\ell_x$-connected, and let \[k\defeq  \min_{x\in \mc{Q}} (m_x+\ell_x)+2.\] Then 
$\mc{Q}$ is, up to homotopy equivalence, obtained from $\mc{P}$ by attaching cells of dimension $\ge k+2$.

Consequently,
\begin{compactenum}[\rm (1)]
\item $\varphi$ induces isomorphisms of homotopy groups up to dimension $k$,
and a surjection in dimension $k+1$.
\item $\mc{P}$ is $k$-connected if and only if 
$\mc{Q}$ is $k$-connected.
\end{compactenum}
\end{lem}

\begin{proof}
Let us consider the poset $M_\varphi$ whose ground set is the disjoint union of the elements of $\mc{P}$ and $\mc{Q}$, and where we define
\begin{compactitem}[$\bullet$]
	\item for $q, q'\in \mc{Q}\subset M_\varphi$: we have $q\le q'$ in $M_\varphi$ if and only if $q\le q'$ in $\mc{Q}$, 
	\item for $p, p'\in \mc{P} \subset M_\varphi$: we have $p\le p'$ in $M_\varphi$ if and only if $p\le p'$ in $\mc{P}$, and
	\item for $p \in \mc{P} \subset M_\varphi$ and $q\in \mc{Q}\subset M_\varphi$: we have $p\le q$ in $M_\varphi$ if and only if $\varphi(p)\le q$ in $\mc{Q}$.
\end{compactitem}
The poset 
$M_\varphi$ triangulates the mapping cylinder of $\varphi$, and therefore strongly deformation retracts to $\mc{Q}$. 
Let $\wt{\varphi}$ denote the inclusion map $\mc{P}\hookrightarrow M_\varphi$. Then for every $x\in \mc{Q}\subset M_\varphi$ we have the isomorphisms
\[\varphi^{-1}(\mc{Q}_{\le x})\cong \wt{\varphi}^{-1}(\mc{Q}_{\le x}) \quad \text{and} \quad \mc{Q}_{> x}\cong (M_\varphi)_{> x}.\]

The key observation now is that
we can obtain $\mc{P}$ from $M_\varphi$ by removing the elements of $\mc{Q}\subset M_\varphi$ 
one by one, until only $\mc{P}$ is left. We do so in an increasing fashion, removing the elements from bottom to top. 

To make this precise, let $\mc{I}$ denote any poset $\mc{P}\subsetneq \mc{I}\subset M_\varphi$. Let furthermore $\mu$ denote a minimal element of $\mc{I}{{\setminus}} \mc{P}$, such that $\mc{I}_{\ge \mu}=(M_\varphi)_{\ge \mu}=\mc{Q}_{\ge \mu}$: no element greater than $\mu$ has yet been deleted from $(M_\varphi)_{\ge \mu}$. 

Now $\Lk_\mu \mc{I}\cong \mc{Q}_{>\mu}\ast {\varphi}^{-1}(\mc{Q}_{\le \mu})$ is $k$-connected by assumption and Lemma~\ref{lem:join}. Hence $\mc{I}$ is obtained from $\mc{I}{\setminus}\{\mu\}$ by successively attaching cells of dimension $\ge k+2$, by Lemma~\ref{lem:lk}. By extension, $M_\varphi\simeq \mc{Q}$ is obtained from $\mc{P}$ by successively attaching cells of dimension $\ge k+2$. The first claim follows,
and this implies the other two.
\end{proof}


\section{Filtered geometric lattices}\label{sec:pos_sum}
This section is devoted to proving the conjecture of Mikhalkin and Ziegler
\cite{MZ} about the lattice of flats of a weighted matroid. We assume familiarity with the basic properties of matroids and geometric lattices, see
\cite{Oxley} and for the homological aspects \cite{BjHomShell}.

Let $M$ denote a matroid on the ground set $[n]\defeq \{1,2,\cdots,n\}$. As a general convention, we shall assume that all matroids are loopless. A \emph{weight} $\omega=(\omega_{1},\omega_{2},\cdots,\omega_n)$ on $M$ is any vector in~$\mathbb{R}^{[n]}$. If $\sigma$ is any subset of $[n]$, and $\mbf{1}_\sigma$ is its characteristic vector, then we set 
\[\omega\cdot \sigma\, \defeq \, \omega\cdot \mbf{1}_\sigma\, =\, \sum_{e\in \sigma} \omega_e.\]
A weight is \emph{generic} if $\omega\cdot \sigma\neq 0$ 
for all proper subsets $\emptyset \subsetneq \sigma\subsetneq [n]$. 
If $\LL=\LL[M]\defeq \widehat{\LL}{{\setminus}} \{\widehat{0},\widehat{1}\}$ is the proper part of the lattice of flats $\widehat{\LL}$ of $M$, and $t$ is any real number, then we use $\LL^{>t}$ to denote the 
subset of $\LL$ consisting of elements $\sigma\in \LL$ with $\omega\cdot \sigma>t$.
We will refer to the posets (partially ordered sets) of the form $\LL^{>t}$ as \emph{filtered geometric lattices}. Note that these posets are not lattices in general, let alone geometric lattices.
Nevertheless, we establish the following theorem, which is the main result of this section.

\begin{mthm}\label{mthm:pos_sum_geom_latt_cm}
Let $\LL$ be the lattice of flats of a matroid of rank $r\ge 2$, and let $\omega$ be a generic weight on its atoms. Let $t$ denote any real number with $t\le \min\{0,\omega\cdot [n]\}$. Then $\LL^{>t}$ is homotopy Cohen--Macaulay of dimension $r-2$, and is in particular $(r-3)$-connected.
\end{mthm}

The proof of this result is articulated in a few steps. We start from homotopy
information available for free matroids,
and from this we deduce information concerning $\LL^{>t}$,
using our generalization of 
Quillen's ``Theorem A'' (Lemma~\ref{mlem:quillen}).
As an immediate corollary, we obtain what can reasonably be called
 a ``Lefschetz Section Theorem for matroids'':

\begin{mcor}\label{mthm:rel_pos_sum_geom_latt_cm}
Let $\LL$ denote the lattice of flats of a matroid of rank $r\ge 2$, and let $\omega$ 
be any generic weight on its atoms. Let $t, t'$ be any pair of real numbers with $t'<t\le \min\{0,\omega\cdot [n]\}$. Then $\LL^{>t'}$ is obtained from $\LL^{>t}$ by attaching cells of dimension $r-2$.
In particular, $(\LL^{>t'}, \LL^{>t})$ is homotopy Cohen--Macaulay of dimension $r-2$. 
\end{mcor}

\begin{proof}
This follows directly from Theorem~\ref{mthm:pos_sum_geom_latt_cm}, the long exact sequence of relative homotopy groups, and Lemma~\ref{lem:lk}. 
\end{proof}

It is instructive to consider the case when $\omega\cdot [n]=0$ and $t=0$.
Then the geometric lattice splits into two parts $\LL^{>0}$ and $\LL^{<0}$,
which are both homotopy Cohen--Macaulay and of the same dimension. Moreover, the union of both $\LL^{>0}$ and $\LL^{<0}$ can be thought of as the complement of an affine hyperplane in $\LL$ via the Bergman fan, compare Section~\ref{ssec:pos}.

For full geometric lattices it follows from Cohen-Macaulayness
and the work of Rota \cite{Rota} on the 
M\"obius function that the number of $(r-2)$-spheres in the wedge is
strictly positive. This is not true for filtered geometric lattices.
For example, if there is exactly one positive weight $\omega_i > 0$
then $\LL^{>0}$ is contractible. However, the following relative
information is immediate in the general case:

\begin{cor}
If $t'<t\le \min\{0,\om\cdot [n]\}$, then \[\dim(H_{r-2}(\LL^{>t})) \le \dim (H_{r-2}(\LL^{>t'})).\]
\end{cor}


\subsection{Preliminaries to the proof}

Let us first observe a general heredity property of filtered geometric lattices that we will use repeatedly for purposes
 of induction. 

\begin{lem}\label{lem:intervals}
Let $(\LL, r, n, \om, t)$ be as in Theorem~\ref{mthm:pos_sum_geom_latt_cm}. Let $(\LL^{>t; \, \omega\,})_{(\si, \tau)}$ be any open interval in $\LL^{>t}$.
Then for $t'=t-\omega\cdot{\si}$ and $\omega'=\omega_{|(\tau-\sigma)}$ we have
\[
(\LL^{>t; \, \omega\,})_{(\si, \tau)} = (\LL_{(\si, \tau)})^{>t'; \om'}. 
\]
\end{lem}

\begin{proof}
Consider first the case when $\si=\emptyset$. Then
$\LL_{(\si, \tau)} =\LL_{<\tau}$ is the lattice of flats of the matroid $M'$ of rank $\rk(\tau)$ 
obtained as the restriction of $M$ to ${\tau}$.
Therefore, $\LL^{>t}_{<\tau}\cong \LL[M']^{>t}$, where $M'$ 
is endowed with the weight given by the restriction $\omega_{|\tau}$ of $\omega$ to the set $\tau$.

Next, suppose that $\tau=[n]$. Then $\LL_{(\si, \tau)} = \LL_{> \si}$ is isomorphic to the lattice of flats of the rank $(r-\rk({\si}))$ matroid $M''$ 
obtained as the contraction of ${\si}$ in $M$.
Moreover, if $M''$ is endowed with weight $\omega_{|[n]{{\setminus}} \si}$ then
\[\LL_{> \si}^{>t}\cong \LL[M'']^{>(t-\omega\cdot{\si})}.\]
Since
$(\LL_{<\tau})_{>\si}=\LL_{(\si, \tau)}$ the general
result is obtained from these two special cases.
\end{proof}

The fact that all maximal chains in 
$\LL^{>t}$ have equal length $r-2$ is a direct consequence.

\begin{lem}\label{lem:pure}
$\LL^{>t}$ is pure and $(r-2)$-dimensional.
\end{lem}
\begin{proof}
For rank $r=2$ the statement boils down to saying that 
$\LL^{>t}$ is nonempty.
Suppose that this were not the case. Then $\om_i\le t$ for all $i$, implying that
$t\le \om\cdot [n]\le tn$, which is impossible if $t<0$.
The case when $t=0=\om\cdot [n]$ is clear.

A proof by induction on rank now follows easily
from Lemma~\ref{lem:intervals} by considering intervals 
$\LL^{>t}_{> \si}$ where $\si$ is an atom.
\end{proof}

\subsection{Free matroids.} We begin with the following strengthening of Theorem~\ref{mthm:pos_sum_geom_latt_cm}
for the special case of free matroids, that is, matroids where all sets are independent.

We reserve the notation $\BB=\BB[n]$ for the proper part of the lattice of flats of the free matroid on $n$ elements. It coincides with the proper part of the Boolean lattice $\wh{\BB}=2^{[n]}$ of subsets of $[n]=\{1,\cdots,n\}$, that is, $\BB =2^{[n]}{{\setminus}} \{\emptyset, [n]\}$.

\begin{thm}\label{thm:bc}
Let $\omega$ denote any generic weight on $[n]$, and suppose that 
$t\le \min\{0,\omega\cdot [n]\}$.
Then $\BB^{>t}$ is shellable and $(n-2)$-dimensional.
In particular, it is homotopy Cohen--Macaulay. 
\end{thm}

\begin{proof} 
We use the method of lexicographic shellability \cite{Bj1, Cortona}. 
We can assume that $\om_i \neq \om_j$ for $i\neq j$. This can always be achieved 
by a small perturbation of the weight vector $\om$ 
that does not change ${\BB}^{>t}$.

To each covering edge $(\si, \tau)$ of $\wh{\BB}$ we assign the real number 
$\la(\si, \tau)\defeq  \om\cdot (\tau{{\setminus}} \si)=\om\cdot \tau-\om \cdot\si)$. This edge labelling induces a 
labelling of the maximal chains of $\wh{\BB}^{\hspace{0.08 em} >t}$. We know from 
Lemma~\ref{lem:pure} that these chains are all of cardinality $n+1$ (including 
the top and bottom elements $\emptyset$ and $[n]$). The label $\la(m)$ of a 
maximal chain $m$ is simply the induced permutation of the coordinates of the 
weight vector $\om$.

There is a unique maximal chain $\ol{m}$ in $\wh{\BB}$ with the property that 
the labels form a decreasing sequence. After relabelling this is 
\[\la(\ol{m})=(\om_1 > \om_2 > \cdots > \om_n ).\]

We have that \[\emptyset \in \wh{\BB}^{\hspace{0.08 em} >t} \;\Lrarr\; t< 0 
\quad \mbox{ and } \quad [n] \in \wh{\BB}^{\hspace{0.08 em} >t} \;\Lrarr\; t< 
\om\cdot [n],\] so the hypothesis $t\le \min \{ 0, \om\cdot [n]\}$ implies that 
both endpoints of the chain $\ol{m}$ belong to $\wh{\BB}^{\hspace{0.08 em} 
>t}$. From this follows that the entire chain $\ol{m}$ is in~$\wh{\BB}^{\hspace{0.08 em} >t}$, as is easy to see. Also, this chain is 
lexicographically first among the maximal chains in $\wh{\BB}$, and so also in 
$\wh{\BB}^{\hspace{0.08 em} >t}$.

Similar reasoning can be performed locally at each interval $(\mu, \nu)$ to 
prove the existence of a unique decreasingly labelled maximal chain in $(\mu, 
\nu)$ which lexicographically precedes all the other maximal chains in that 
interval. This completes the verification of the conditions for lexicographic 
shellability. \end{proof}

\begin{rem} The conclusion of the theorem can be sharpened to state that 
$\BB^{>t}$ is PL-homeomorphic to a ball or a sphere. Some aspects of this 
additional information are discussed in \cite{Cortona}, it will not be needed 
here.\end{rem}

\subsection{Connectivity and Cohen--Macaulayness.} We continue the 
proof of Theorem~\ref{mthm:pos_sum_geom_latt_cm} by establishing the degree of 
connectivity for $\LL^{>t}$.

\begin{thm}\label{thm:filtgl}
Let $(\LL, r, n, \om, t)$ be as in the statement of Theorem 
\ref{mthm:pos_sum_geom_latt_cm}. Then $\LL^{>t}$ is $(r-3)$-connected.
\end{thm}

\begin{proof}
We prove the theorem by induction on the cardinality $n=|M|$, the case $n=1$ being trivial.
With the intent of applying a Quillen-type fiber argument, let us consider the inclusion map $\varphi:\LL^{>t}\hookrightarrow \BB^{>t}$. We must  analyse the fibers $\varphi^{-1} (\BB^{>t}_{\ge x})$ and the lower ideals 
$\BB_{<x}^{>t}$ for all $x\in \BB^{>t}$. 

We have that $t< \om \cdot x$, since $x\in\BB^{>t}$, and 
$t\le \min\{0,\omega\cdot [n]\}$. Hence, $t\le \min\{0,\omega\cdot x\}$, and it follows from
Lemma~\ref{lem:intervals} and Theorem~\ref{thm:bc} that the posets 
$\BB_{<x}^{>t} \cong \BB[x]^{>t}$ are $(|x|-3)$-connected. 

It remains to consider the fibers $\varphi^{-1} (\BB^{>t}_{\ge x})$.
Let ${\kappa}:\BB\rightarrow \LL$ denote the matroid closure map 
${S}\mapsto \bigvee_{e\in {S}} e$, and let $x$ be any element in $\BB^{>t}$.
Then,
 \[\varphi^{-1} (\BB^{>t}_{\ge x}\ )=\ \LL^{>t}_{\ge x}\ =\ \LL^{>t}_{\ge \kappa(x)}.\] 
If $\kappa(x)\in \LL^{>t}$, the fiber is a cone, and hence contractible. 

If $\kappa(x)\notin \LL^{>t}$, then by the induction assumption and Lemma~\ref{lem:intervals}, $\LL^{>t}_{\ge x}$ is 
$(\dim \LL_{\ge \kappa(x) }^{>t}-1)$-connected.

We have shown that the fiber
$\varphi^{-1} (\BB^{>t}_{\ge x})$ is 
$(\dim \LL_{\ge \kappa(x) }^{>t}-1)$-connected and the ideal 
$\BB_{<x}^{>t}$ 
is  $(|x|-3)$-connected,
for all $x\in \BB^{>t}$.
Hence, by the Fiber Lemma~\ref{mlem:quillen}, the inclusion map $\varphi$ yields an isomorphism of homotopy groups up to dimension $k$, and a surjection in dimension $k+1$, where
\begin{align*}
k\ \defeq \ & \min_{\substack{x\in \BB^{>t}\\ \kappa(x)\notin \LL^{>t}} } (\dim (\LL^{>t}_{\ge \kappa(x)})+|x|)-2 
\end{align*}
Now, for $x\in \BB^{>t}$ with $\kappa(x)\notin \LL^{>t}$, we have
\begin{align*}
 \dim \LL^{>t}_{\ge\kappa(x)}+|x|-2 \ 
\ge&\ \dim \LL^{>t}_{\ge_\kappa(x)}+\dim(\LL^{>t}_{\le \kappa(x)})-1\\
=&\ \dim \LL^{>t} -1\\
=&\ r-3
\end{align*}
Hence, $k\ge r-3$, and since $\BB^{>t}$ is $(r-3)$-connected, so is $\LL^{>t}$.
\end{proof}

We can now finish and prove the homotopy Cohen--Macaulay property, 
which demands that we show the purity of $\LL^{>t}$ and that each interval is connected up to 
its dimension minus one.

\begin{proof}[\textbf{Proof of Theorem~\ref{mthm:pos_sum_geom_latt_cm}}]
Let $(\LL^{>t})_{(\si, \tau)}$ be an open interval.
We know from Lemma~\ref{lem:pure} that its
order complex has dimension $\rk(\tau)-\rk(\si)-2$ and from 
Lemma~\ref{lem:intervals} and 
Theorem~\ref{thm:filtgl} that 
it is $(\rk(\tau)-\rk(\si)-3)$-connected.
\end{proof}


\section{Remarks, examples and open problems.}
\label{ssc:rem1}

\subsection{Connectivity.}\label{ssc:crime}

Theorem~\ref{mthm:pos_sum_geom_latt_cm}
has some content of a combinatorial nature which deserves explicit statement.

Let $\LL$ be the lattice of flats of a matroid of rank $r\ge 3$, 
and let $\LL_a =\{x\in\LL\mid \text{rank(x)}=a\}$. Then, for integers 
$0<a<b<r$ let $G_{a,b}$ be the bipartite graph with vertex set
$\LL_a \cup \LL_b$ and edges induced by set containment, that is, $G_{a,b}$
is the Hasse diagram of $\LL$ restricted to ranks $a$ and $b$.

Now, let $\omega$ be any generic weight on the atoms
$\LL_1$ of $G_{a,b}$, summing to $0$. 
As before, via summation this induces a weight on each flat in $\LL$, 
thus splitting each rank level $\LL_a$ into a positive and a negative 
part $\LL_a = \LL_{a}^{+} \cup \LL_a^-$. In particular, 
the vertices of our bipartite graph  $G_{a,b}$ are either positive or
negative. Let  $G_{a,b}^+$ denote the subgraph induced on the
positive vertices, and similarly for $G_{a,b}^-$.

\begin{mcor}\label{mthm:crime}
Under the stated circumstances, both graphs
$G_{a,b}^+$ and $G_{a,b}^-$ are connected.
\end{mcor}
\begin{proof}
As a one-dimensional subcomplex of
the full order complex of $\LL$,  
the bipartite graph $G_{a,b}^+$ 
is obtained by rank-selection. Rank-selection 
preserves Cohen-Macaulayness, see e.g. 
\cite[Theorem 11.13]{BTM}
for references to several original sources 
for this fact. Since for one-dimensional
complexes Cohen-Macaulayness is equivalent to being connected,
the proof is complete.
\end{proof}

It is worth noticing that we may initially place the weight  on the elements of any 
rank level $\LL_a$. As an instance of the {\it finite Radon transform} 
(see \cite{Kung})
this determines the weight function on the set of atoms
$\LL_1$, and from there the weight function on all of $\LL$.

\vspace{3mm}

The 
$(a,b)=(1,2)$ case of Corollary~\ref{mthm:crime}
can be illustrated by the following mock application to crime prevention.

Suppose we are dealing with a crime-ridden city
in which every two streets cross once and only once. 
Each street  
has  been given a  ranking number,
reflecting how safe it is to walk along that street. To normalize the grading,
the average rank has been set to be zero. Streets with
positive rank are considered safe, those with negative rank are not.
Furthermore, a street corner in the city is considered
safe to stop and turn at if the average rank of all streets intersecting in that corner is positive,
otherwise it is dangerous.
The question is: Is it possible to walk from any safe street 
to any other safe street without ever turning at a dangerous corner or walking along 
a dangerous street? 

The answer is yes (see Figure~\ref{fig:crime}), since 
we are dealing with a rank $3$ matroid whose atoms are given by the streets, and whose rank $2$ flats are given by the corners, and hence the associated graph  $G_{1,2}$, restricted to positive elements to obtain $G_{1,2}^+$, is connected. 

\begin{figure}[htb] 
\centering 
 \includegraphics[width=0.7\linewidth]{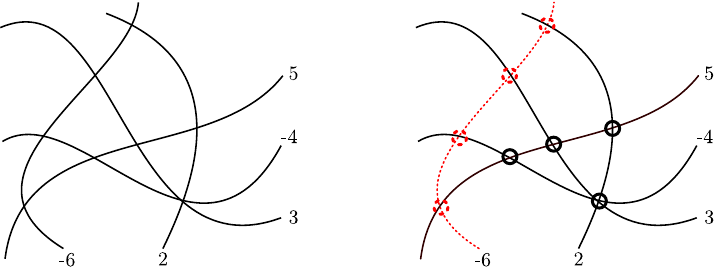}
\caption{\small A crime-ridden town, and the safe and unsafe corners inside it.} 
 \label{fig:crime}
\end{figure}

Note that also
the crooks have an advantage: they can access any unsafe street from their 
hideout without turning at a safe corner, or passing along a safe street, 
as also  $G_{1,2}^-$ is connected.

\subsection{Shellability.}
It remains to be seen whether our understanding of the combinatorial structure of filtered geometric lattices can be improved further. One natural question is whether a stronger statement can be made about the geometry of filtered lattices.

\begin{prb}
Let $(\LL, r, n, \om , t)$ be as in Theorem~\ref{mthm:pos_sum_geom_latt_cm}. Is it true that $\LL^{>t}$ is shellable?
\end{prb}

A positive answer would generalize earlier work of Bj\"orner \cite{Bj1} showing that every full geometric lattice is shellable, and by Wachs and Walker \cite{WachsWalker} dealing with the case when 
the weight $\omega$ has exactly one negative entry.

\subsection{Combinatorial Morse Theory.}\label{ssc:Forman} For an alternative approach to the conjecture of Mikhalkin and Ziegler one can use the combinatorial Morse theory of Forman \cite{FormanADV}. Intuitively speaking, combinatorial Morse theory is an incremental way to decompose a simplicial complex.
It enriches Whitehead's notion of cell collapses \cite{Whitehead} by the notion of \emph{critical cells}, which behave analogously to critical points in classical Morse Theory. The result is:

\begin{mthm}\label{mthm:pos_sum_geom_latt_dm}
Let $(\LL, r, n, \om)$ be as in Theorem~\ref{mthm:pos_sum_geom_latt_cm}, with $\omega\cdot[n]=0$. 
Then, there is a collection $\mr{C}$ of critical $(r-2)$-cells such that $\LL^{>0}-\mr{C}$ simplicially collapses to a point.
In particular, $\LL^{>0}$ is $(r-3)$-connected.
\end{mthm}

In comparison with Theorem~\ref{mthm:pos_sum_geom_latt_cm}, this result requires a stronger assumption (the total weight of $\omega$ is $0$).
But it also has a stronger conclusion, since it describes the combinatorial structure of $\LL^{>0}$ and not only its topological type. For the proof, one uses Alexander duality of combinatorial Morse functions as introduced in \cite{A}. This can be exploited to prove Theorem~\ref{mthm:pos_sum_geom_latt_dm} and the analogous theorem for the complement of $\LL^{>0}$ in $\BB$ by a common induction.

\subsection{General filtered geometric lattices.} Towards a more complete understanding of filtered geometric lattices, it remains to consider the case when the filtration parameter does not satisfy $t\le\min\{0,\om\cdot [n]\}$.

\begin{prb}
Characterize the topology of $\LL^{>t}$ for general $t$.\end{prb}

It would seem natural to expect that $\LL^{>t}$ is always {\em sequentially Cohen--Macaulay}, a notion introduced by Stanley 
to generalize Cohen--Macaulayness to nonpure complexes, cf.\ \cite{SCCA, BWW2}. This, however, is not the case.

\begin{ex}
Let us consider the matroid $M$ on ground set $[7]$, endowed with lattice of flats \[\LL\defeq  \{\{1\},\{2\},\{3\},\{4\}, \{5\},\{6\},\{7\}, \{1,2\}, \{6,7\}, \{1,3,6\}, \{1,4,7\}, \{2,4,6\}, \{2,5,7\}, \{3,4,5\}\},\]
see Figure~\ref{fig:matroid}.
Let us furthermore consider the weight $\om=(1,1,-3,-3,-3,1,1)$. Then 
\[\LL^{>0}=\{\{1\},\{2\},\{6\},\{7\}, \{1,2\}, \{6,7\}\},\]
which consists of two disjoint 1-dimensional complexes. Hence, $\LL^{>0}$ is not sequentially connected, and in particular not sequentially Cohen--Macaulay.
\begin{figure}[htb] 
\centering 
 \includegraphics[width=0.28\linewidth]{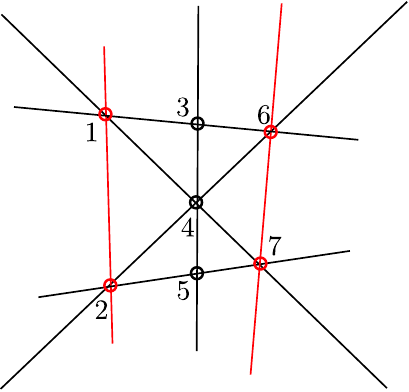} 
\caption{\small The matroid $M$, its proper flats and the filtered geometric lattice $\LL^{>0}$ in red.} 
 \label{fig:matroid}
\end{figure}
\end{ex}

We have established that $\BB$ is obtained from $\BB^{>t}$ (Theorem~\ref{thm:bc}), and that $\BB^{>t}$ is obtained from $\LL^{>t}$ (Theorem~\ref{thm:filtgl}), by successively attaching cells of dimension $\ge r-2$. One can reverse the reasoning to prove the following theorem:

\begin{thm}\label{thm:attaching_cells_complement}
Let $(\LL, r, n, \om, t)$ be as in Theorem~\ref{mthm:pos_sum_geom_latt_cm}. 
Then $\BB-\LL$ is obtained from $\BB^{\le t}-\LL^{\le t}$ by attaching cells of dimension $\le n-r-1$.
\end{thm}

The proof is entirely analogous to that of Theorem~\ref{thm:filtgl}, and will be left out here. We notice, however, several facts: 

(1) By Lemma~\ref{lem:deletion_in_simplicial} and Alexander/Poincar\'e-Lefschetz duality, we have an isomorphism
\[H_i(\BB-\LL,\BB^{\le t}-\LL^{\le t})\cong H^{n-i-3}(\LL^{> t}).\]
Theorem~\ref{thm:attaching_cells_complement} therefore provides an alternative proof for the homology version of Theorem~\ref{thm:filtgl} (and vice versa). 

(2) Furthermore, if $n-r\neq 2$, Theorems~\ref{thm:filtgl} and~\ref{thm:attaching_cells_complement} are equivalent by well-known general position arguments together with the aforementioned duality.

(3) It is possible to give a common proof of Theorem~\ref{thm:filtgl} and Theorem~\ref{thm:attaching_cells_complement}, using combinatorial Morse theory and Alexander duality of combinatorial Morse functions, cf.\ Section~\ref{ssc:Forman}.

(4) The pair $(\BB-\LL,\BB^{\le t}-\LL^{\le t})$, 
being the complement of an $(r-2)$-dimensional complex $\LL$ in the $(n-3)$-connected, $(n-2)$-dimensional pair $(\BB,\BB^{\le t})$, is $(n-r-2)$-connected by classical general position arguments. Together with the information that the pair is of dimension $\le n-r-1$, we immediately obtain the following
more precise version of Theorem~\ref{thm:attaching_cells_complement}:

\begin{cor}\label{cor:attaching_cells_complement}
Let $(\LL, r, n, \om, t)$ be as in Theorem~\ref{mthm:pos_sum_geom_latt_cm}. 
Then $\BB-\LL$ is obtained from $\BB^{\le t}-\LL^{\le t}$ by attaching $(n-r-1)$-dimensional cells.
\end{cor}

In particular, we can extend the results on the homotopy type of geometric lattices to their complements.

\begin{mcor}\label{mcor:Comp_Bergman_Fan}
Let $\LL$ denote the lattice of flats of a matroid of rank $r\ge 2$, and let $\BB$ denote the proper part of the Boolean lattice on the ground set of $M$. Then $\BB-\LL$ is homotopy equivalent to a wedge of spheres of dimension $|M|-r-1$.
\end{mcor}

\subsection{An efficient model for complements of geometric lattices.} While we now understand, from a homotopical point of view, the complement of a geometric lattice in the Boolean lattice on the same support, it might be desirable to have a more explicit model available. For this purpose, we can use an idea similar to Salvetti \cite{Salvetti} and Bj\"orner--Ziegler~\cite{BjZie}, who described models for the complement of subspace arrangements. Throughout this remark, we use $(M,\LL, \BB, r, n)$ as in the previous sections.

A naive model for the complement $\BB {{\setminus}} \LL$ of $\LL$ in $\BB$ is clearly given by the complex $\BB-\LL$,
as shown by Lemma 
\ref{lem:deletion_in_simplicial}.
However, the complex $\BB-\LL$ can be of dimension up to $n-2$, while $\BB-\LL$ only has the homotopy type of a complex of dimension $\le n-r-1$, so that this model can be considered quite wasteful.

To obtain a more efficient model for a matroid $M$ on the ground set $[n]$, let $\NS$ denote the poset of 
\emph{non-spanning} proper subsets of $M$ ordered by inclusion. In other words, $\NS$ consists of the subsets $\sigma$ of $[n]$ with 
matroid rank $\rk (\sigma)<\rk (M)$. 
Now, as mentioned in Section~\ref{ssc:ct}, the matroid closure map 
\begin{align*}
\kappa:\NS\ \ &\rightarrow\ \ \LL,\\                                                                    
 x\ \ &\mapsto \ \ \bigvee x                                                                   
\end{align*}
 deformation retracts $\NS$ to the geometric lattice $\LL$ in $\BB$. We obtain:
\begin{thm}\label{thm:salvetti_efficient}
With $(\LL, \BB, r, n)$ as above, we have
\[\BB {{\setminus}} \LL\ \simeq\ \BB - \LL\ \simeq \ \BB - \NS.\]
Moreover, $\BB - \NS$ is an efficient model, in the sense that $\dim (\BB - \NS)\le n-r-1$.
\end{thm}

\begin{proof} It remains only to verify the claim on the dimension; this follows immediately once we notice that every element of $\BB$ of cardinality $\le r-1$ is non-spanning.\end{proof}

It follows from the work of Rota \cite{Rota}
that the dimension bound of Theorem~\ref{thm:salvetti_efficient} is tight.

\subsection{Matroid duality is Alexander duality.}
The dimension of $\BB - \NS$ is bounded above by $n-r-1$, which coincides with the rank of the dual matroid $M^\ast$ of $M$. This suggests a connection between $\BB - \NS\simeq \BB - \LL$ and $\LL[M^{\ast}]$. Indeed, as was 
pointed out in \cite[Exercise 7.43, p. 278]{BjHomShell}, 
such a connection is provided by combinatorial 
Alexander duality (see e.g. \cite{BjT}). We have,
$$ \BB - \NS\\
\simeq\ \ \{[n]{{\setminus}} \sigma: \sigma\ \text{spanning in}\ M\}
\cong\ \ \{\tau: \tau\ \text{independent in}\ M^\ast\}
$$
The second complex is precisely the combinatorial Alexander dual of $\NS$, and the last isomorphism follows from standard matroid duality.

It is known that also the poset $\mc{I}$ of independent sets
of a matroid $M$ is shellable, and in particular $(r-2)$-connected, cf.\ \cite{BjHomShell}. 
Combined with the previous remark, this provides an alternative proof of Corollary~\ref{mcor:Comp_Bergman_Fan}.

\part{Lefschetz theorems for smooth tropical varieties}


\section{Geometry and combinatorics of polyhedral complexes and smooth tropical varieties.}\label{sec:polyhedra}

In this and the following section we review and develop the foundations of the geometry of polyhedral and tropical spaces, including their study via stratified Morse Theory.

\subsection{Open polyhedra and topology of restrictions}\label{ssec:restref} Let $A$ and $B$, $B \subset A$, denote two polyhedral complexes, such that for every face $b$ of $B$ there exists a unique minimal face $a$ with the property that $a\in A,\ a\notin B,\ b\subset a$. Then ${a}$ is the \emph{cofacet} of $b$ in $A$, and $O=A{{\setminus}} B$ is an \emph{open polyhedral complex}.
This condition implies in particular that a regular neighborhood of $B$ in $A$ is PL homeomorphic to $B\times [0,1]$, and that $A$ collapses (in the sense of Whitehead) onto $A-B$, cf.\ \cite{Whitehead}. The faces of $O$ are the faces of $A$, minus the faces of~$B$.

In general, there is little relation between a polyhedral complex and its restrictions. However, the following observation for restrictions of polyhedral complexes is useful to keep in mind for applications of stratified Morse theory.

\begin{prp}\label{prp:restriction_gives_def_retract}
Let $X$ denote a compact polyhedral complex in $\RR^d$, and let $C$ denote the complement of some open, convex set $K$ in $\RR^d$. Then $X\cap C=X{{\setminus}} K$ deformation retracts onto $\RS{X}{C}$.
\end{prp}

\begin{proof}
If $A$, $B$ are convex sets (closed and open, respectively) in $\RR^d$ with a point of intersection $x$, then $A{{\setminus}} B$ deformation retracts onto $\partial A{{\setminus}} B$ via restriction of the radial projection 
\begin{align*}
A{{\setminus}} \{x\}\ \ &\longrightarrow\ \ \partial A\\
y\ \ &\longmapsto\ \ (x + \pos (y-x)	)\cap \partial A.
\end{align*}
We can now argue by induction on the dimension of faces of $X$: We claim that if $\sigma$ is any facet of $X$ that intersects $K$, then $\sigma$ deformation retracts onto $\partial \sigma{{\setminus}} K$.
 Therefore \[X\cap C=((X-\sigma) \cap C) \cup (\sigma \cap C)\] deformation retracts onto $(X-\sigma) \cap C$. With this procedure we can iteratively remove all faces of $X$ not in $C$ by deformation retractions. The claim follows. 
\end{proof}

 \subsection{Tangent fan, stars and links.}\label{sec:tangent} Let $X\subset \RR^d$ be a polyhedral space, and let $p\in X$ be any point. We use $\TT_p X$ to 
denote the or tangent space of $X$ at~$p$, and $\TT^1_p X$ is the restriction of $\TT_p X$ to unit vectors. As for polyhedra the tangent space is a fan, we shall simply refer to it as a tangent fan going forward. If $Y$ is any subspace of $X$, then $\RN_{(p,Y)} X$ denotes the subspace of the tangent fan spanned by vectors orthogonal to $\TT_p Y \subset \TT_p X$, and we define $\RN^1_{(p,Y)} X\defeq  \RN_{(p,Y)} X \cap \TT^1_p Y$.

If $X$ is polyhedral, and $\sigma$ is any face, then $\TT_p X$, $\TT^1_p X$, $\RN_{(p,\sigma)} X$ and $\RN^1_{(p,\sigma)} X$ are, up to ambient isometry, independent of the choice of $p$ in the relative interior of $\sigma$; we therefore omit $p$
whenever feasible and write simply $\TT_\sigma X$ etc.

Now, let $X$ be a polyhedral complex and let $\sigma$ be any face of $X$. The \emph{star} of $\sigma$ in $X$, denoted by $\St_\sigma X$, is the minimal subcomplex of $X$ that contains all faces of $X$ containing $\sigma$.

Let $\tau$ be a face of a polyhedral complex or fan $X$ containing a face $\sigma$, and assume that $\sigma$ is nonempty. Then the set $\RN^1_{\sigma} \tau$ of unit tangent vectors in $\RN^1_{\sigma} X$ pointing towards $\tau$ forms a spherical polytope in~$\RN^1_{\sigma} X$. 
The collection of all polytopes in $\RN^1_{\sigma} X$ obtained this way forms a polyhedral complex, denoted by 
$\Lk_\sigma X$, the \emph{(combinatorial) link} of $\sigma$ in $X$. We set $\Lk_\emptyset X\defeq  X$. Motivated by this, we shall also sometimes call $\RN^1_{\sigma} X$ the \emph{(geometric) link} of $\sigma$ in $X$; both 
types of links have the same underlying space, but $\Lk_\sigma X$ enjoys additionally a combinatorial structure.

\subsection{Morse functions on polyhedral and stratified spaces.}
If $\sigma$ is any face of a polyhedral complex, then we call its relative interior $\sigma^\circ$ a \emph{stratum} (or \emph{cell}).
Let now $\wt{f}:S\rightarrow \RR$ denote a function whose domain $S\subset \RR^d$ is open, and let $X$ denote any polyhedral space in $S$. A \emph{critical point} of $f\defeq \wt{f}_{|X}$ is a critical point of $f_{|\sigma^\circ}$ where $\sigma\in X$ is any face. In other words, $x$ is a critical point of $f$ in $X$ if for the unique stratum $\sigma^\circ$
of $X$ containing it, we have $\nabla{\wt{f}}\perp \TT_x \sigma^\circ$. \emph{Critical values} are the values of critical points under~$f$. 
We call $f$ a \emph{Morse function} on $X$ if 
\begin{compactenum}[\rm (1)]
\item $\wt{f}$ is smooth on $S$ and every stratum of $X$.
\item $f$ is proper, and the critical points of $f$ are finitely many and have distinct critical values.
\item All critical points are nondegenerate, i.e.,\ for every face $\sigma\in X$, and every critical point $x\in \sigma^\circ$, the Hessian of $f_{\sigma}$ at $x$ is non-singular.
\item Every critical point is the critical point in a unique stratum, i.e.,\ for every critical point $x$ of $f$ in a stratum $\sigma^\circ$, and for every proper coface $\tau$ of $\sigma$ in $X$, we have $\nabla{\wt{f}}\not\perp \TT_x \tau$.
\end{compactenum}
For open polyhedral complexes, we simply require that the gradient field is uniformly outwardly oriented at the boundary. Specifically, if $O=A{{\setminus}} B$ is an open polyhedral complex, then the restriction $f$ of $\wt{f}$ to $O$ is a \emph{Morse function} on $O$ if
\begin{compactenum}[\rm (1)]
\item $\wt{f}_{|A}$ is a Morse function on $A$, and 
\item for every $b\in B$, and the unique cofacet $a\in A{\setminus} B$ of $b$ in $A$,
\[ \langle \nu, \nabla \wt{f}(x)\rangle\ <\ 0\]
for the interior normal $\nu$ to $b$ in $a$.
\end{compactenum}
We shall call the last property \emph{uniform outward orientation}.

\subsection{The main lemma of stratified Morse Theory.}\label{ssc:strat} With this, we can state the main lemma of stratified Morse Theory, specialized to polyhedral spaces (i.e.\ underlying spaces of polyhedral complexes). 
\begin{thm}[Goresky--MacPherson {\cite[Part~I]{GM-SMT}}]\label{thm:smt_main_theorem}
Let $X$ denote a polyhedral space, and let $f=\wt{f}_{|X}:X\rightarrow \RR$ be a Morse function on $X$ as above. Then
\begin{compactenum}[\rm (1)]
\item If $(s, t]\subset \R$ is an interval containing no critical values of $f$, then $X_{\le s}\defeq  	f^{-1}(-\infty, s]$ is a deformation retract of $X_{\le t}$.
\item Suppose that $t$ is a critical value of $f$, $x$ the associated critical point and $s< t$ is chosen so that $(s, t]$ contains no further critical values of $f$. Then, the Morse data at $x$ (and therefore the change in topology from $X_{\le s}$ to $X_{\le t}$) is given by the product of tangential and normal Morse data of $f$ at~$x$. 
\end{compactenum}
\end{thm}

\subsection{Convex superlevel sets.} If $\wt{f}$ has convex superlevel sets, we can easily work out the normal and tangential Morse data.
\begin{lem}\label{lem:smt_observation_on_level sets}
With the notation as in Theorem~\ref{thm:smt_main_theorem}, let us assume that for every critical value $t$ of $f$, $\wt{f}^{-1}[t,\infty)$ is closed and convex. 
Let $x$ be a critical point of $f$, let $\sigma^\circ$ denote its stratum and 
assume that $s< t=f(x)$ is chosen so that $(s, t]$ contains no further critical values of $f$. Then 
\begin{compactenum}[\rm (1)]
\item the tangential Morse data at $x$ is given by $(\sigma,\partial \sigma)$, and
\item the normal Morse data at $x$ is given by \[\big(\CO(\RN^1_\sigma (X\cap f^{-1}(-\infty,t])),\RN^1_\sigma (X\cap f^{-1}(-\infty,t])\big).\]
\end{compactenum}
\end{lem}

\begin{proof}
Since $\wt{f}^{-1}[t,\infty)$ is closed, smooth and convex for every critical value $t$, the Morse function $f_{|\sigma^\circ}$ takes a minimum at $x\in \sigma^\circ$. Claim $(1)$ follows. Claim $(2)$ holds regardless of the requirement on superlevel sets, cf.\ \cite[Part~I, Section~3.9]{GM-SMT}.
\end{proof}

\section{Basic notions in tropical geometry.}\label{ssc:trop} Here we give a brief overview of the essentials of tropical geometry; for more information, we refer to \cite{Gathmann, ms, RGST, SpeyerSturmfels}.

Let $\TR\defeq  [-\infty,\infty)=\RR\cup\{-\infty\}$, the tropical numbers. $\TR$ is a semiring endowed with the \emph{(tropical) addition} $\tp:\TR\times \TR \rightarrow\TR$ and \emph{(tropical) multiplication} $\tm:\TR\times \TR \rightarrow\TR$ defined as \[a\tp b\defeq  \max\{a,\, b\}\qquad \text{ and } \qquad a\tm b\defeq  a\, +\, b.\]

The \emph{tropical affine space} $\TR^d$ of dimension $d$ is the space $[-\infty,\infty)^d$. The \emph{fine sedentarity} $\Sed:\TR^d\rightarrow 2^{[d]}$ of a point $x\in \TR^d$ is the set $\{i\in [d]: x_i= -\infty\}$. The \emph{sedentarity} $\sed:\TR^d\rightarrow \NA$ is defined as $\sed(x)\defeq |\Sed(x)|$. A point 
or set of sedentarity $0$ is also called \emph{mobile}. The \emph{mobile part} of a subset $A$ in $\TR^d$ is also denoted by $A_{|\mo}$, and 
we write $A_{|\Sed=I}$ to restrict to the subset of fine sedentarity~$I$.
In particular, we have a decomposition
\[\TR^d=\bigcup_{I\subset [d]} \RR^I\times (-\infty)^{[d]{{\setminus}} I}.\]
We define \emph{tropical projective $d$-space} as
\[\TP^d\defeq  \bigslant{\TR^{d+1}{{\setminus}} (-\infty)^{d+1}}{x\sim \lambda \tm x}.\]
Tropical projective space $\mathbb{TP}^{d}$ can be obtained as a union of $d+1$ copies of tropical affine space $\TR^{d}$, restricted to nonpositive coordinates: If $i\in [d+1]$ is any element, then the set $\wt{S}_i\defeq  \{x\in \TR^{d+1}: x_i=\max_{j\in [d+1]} x_j\}$ projects to the copy of $\TR^{d}_{\le 0}$ spanned by $\{x_j: j\in [d+1]{{\setminus}} \{i\}\}$. The notions of sedentarity and mobility therefore naturally extend to tropical projective space. 

\begin{rem}
We extend notions from affine tropical geometry to projective tropical geometry using this decomposition. Therefore, we restrict the discussion to the affine definitions for the remainder of this introduction to tropical geometry.
\end{rem}

\subsection{Polyhedral spaces in $\TR^d$.}\label{sec:pT} Recall that tropical space $\TR^d$ is stratified into copies $\RR^I\times (-\infty)^{[d]{{\setminus}} I}$, $I\subset [d]$ of euclidean vector spaces.
A $d$-\emph{polyhedron} $P$ in $\TR^d$ is the closure of a polyhedron in $\RR^d$. To ensure that the polyhedra are sufficiently nice at infinity, 
we require that for every face $Q=P\cap (\TR^J\times (-\infty)^{[d]{{\setminus}} J})$, $J\subset [d]$, and every $I\subsetneq J$, we have
\begin{equation}\label{eq:reg}
Q\cap (\RR^I\times (-\infty)^{[d]{{\setminus}} I})\ =\ \varnothing \quad \text{or} \quad \dim (Q\cap (\RR^I\times (-\infty)^{[d]{{\setminus}} I}))\ =\ \dim Q - |J| + |I|.
\end{equation}
A \emph{chamber complex} is a polyhedral space that divides the complement into pointed polyhedra.

We say a chamber complex $H$ is \emph{ample} if for every cell $Q$ (the cell being the relative interior of a face) in the natural stratification of $\mathbb{A}^d$, we have
\begin{compactitem}[$\circ$]
\item $H\cap Q$ is a chamber complex for $Q$, that is, it parts $Q$ into pointed polyhedra, and
\item every such component in the closure of $Q$ (that is, in the closed face corresponding component) is combinatorially equivalent to an orthant.
\end{compactitem} 
Note that in the case of $\mathbb{A}^d\cong\R^d$, $\TR^d$ or $\TP^d$, chamber complexes are automatically ample.
 Moreover, note that the first condition may equivalently be restated as: Every tropical curve intersects $H$.

As this boils down to a condition on the individual components of the complement of $H$, we call the individual components ample as well, that is, 
we say a polyhedron $P$ in $\mathbb{A}^d$ is \emph{ample} if, for every cell $Q$ as above it intersects, the restriction to the cell is a pointed polyhedron, and the relative interior of $P$ intersected with the closure of $Q$ is combinatorially an orthant.

In particular, if $P$ contains a point of sedentarity $[d]{{\setminus}} J$, then its affine span contains rays in direction of the basis vectors $(e_j)_{j\in [d]{{\setminus}} J}$. The affine span of $P$ is therefore generated by the intersection $P\cap (\TR^J\times (-\infty)^{[d]{{\setminus}} J})$ and the vectors $(e_j)_{j\in [d]{{\setminus}} J}$. 

A \emph{polyhedral complex} $\Sigma$ in $\TR^d$ is a collection of polyhedra in $\TR^d$ with the property that the intersection of any two polyhedra is a face of both. The notions of restriction, link, star etc.\ for polyhedral complexes in $\RR^d$ naturally extend to the tropical case.

\subsection{Bergman fans of matroids.} Associated to every matroid $M$ (which, we remind here, we assume to be loopless throughout) is the 
{\em Bergman fan} \cite{ArdilaKlivans, Bergman, SturmfelsBergman}. We identify the elements 
of the ground set $[n]$ of $M$ with a generating circuit of lattice vectors $e_1,\cdots,e_n$ in $\mathbb{Z}^{n-1}\subset\RR^{n-1}$ such that $\sum_{i=1}^n e_i=0$. If $F$ is any subset of $[n]$, we define
$e_F\defeq \sum_{e_i\in F} e_i$. Moreover, if $\mbf{C}=F<G<H<\cdots$ is any chain in $\LL$ (the lattice of flats  of $M$), then \[\pos(\mbf{C})\defeq  \pos\{e_{F},e_{G},e_{H},\cdots\}.\]
The \emph{Bergman fan} $\B(M)$ of $M$ is the fan
\[\B(M)\defeq  \{\pos(\mbf{C}):\mbf{C}<[n] \text{ increasing chain in } \LL\}.\]
We observe a useful property for later:

\begin{lem}[Balancing property]\label{lem:balance}
Let $\mbf{C}$ be a maximal chain in $\LL$, and let $F$ denote any element of $\mbf{C}$. Then $\B$ is balanced with unit (and in particular positive) weights at $\pos(\wt{\mbf{C}})$, where $\wt{\mbf{C}}\defeq  \mbf{C}{\setminus} F$, i.e.,
\[\sum_{\substack{G\in \LL\\ G\cup \wt{\mbf{C}}\ \text{maximal chain}}} e_G\ \in\ \lin(\wt{\mbf{C}}).\]
\end{lem}

\begin{proof} 
This follows at once from the lattice partitioning axiom for the lattice of flats. 
\end{proof}

A second useful observation is to note that every Bergman fan is also locally a Bergman fan:
\begin{lem}\label{lem:intersections}
Let $\B(M)$ denote a Bergman fan, and let $\rho=\rho_F$ denote the ray of a flat $F$ of $M$. Then both $\RN_{\rho} \B$ and $\TT_{{\rho}} \B$ are themselves Bergman fans of matroids.
\end{lem}

Note that this does not hold with respect to the combinatorial product structure on the product of Bergman fans; it is a statement about their underlying spaces only.
We shall more generally observe:

\begin{lem}\label{lem:Bergmanp}
The (underlying space of the) product of two Bergman fans $\B(N),\ \B(L)$ is a Bergman fan.
\end{lem}

\begin{proof}
The desired matroid is obtained as the parallel connection $M$ of $N$ and $L$, cf.\ \cite{BT}. It is obvious that $\B(N) \times \B(L)$ is contained in $\B(M)$. To see the converse, note that both fans are of the same dimension $\rk(N)+\rk(L)-2$, and admit a unique Minkowski weight in degree $\rk(N)+\rk(L)-2$ that is nowhere zero.
\end{proof}

\begin{proof}[\textbf{Proof of Lemma~\ref{lem:intersections}}]
Let $M_{|F}$ denote the restriction of $M$ to $F$, and $M/F$ the contraction of $F$ in $M$. Then $\RN_{\rho} \B$ is obtained as the product of $\B(M_{|F})$ and $\B(M/{F})$.
Finally, $\TT_{{\rho}} \B$ is obtained as the product of $\B(\BB[2])$ with $\RN_{\rho} \B$.\end{proof}

\subsection{Smooth tropical varieties.} Smooth tropical varieties were introduced by Mikhalkin, although it took time for them to appear in writing (see for instance \cite{Shaw} for a thorough discussion). An \emph{integral affine map} $\varphi: \TR^{n}\rightarrow\TR^{m}$ is a map that arises from a well-defined extension of an {integral affine map} $\wt{\varphi}: \RR^{n}\rightarrow\RR^{m}$, which in turn is defined as the composition of an integral linear map and an {arbitrary} translation. 

An \emph{abstract smooth tropical variety} of dimension $n$ is an abstract polyhedral complex $X$ with charts $(U_\alpha, \Phi_\alpha)$; $\Phi_\alpha:U_\alpha\rightarrow V_\alpha \subset \TR^{N_\alpha}$ such that
\begin{compactenum}[(1)]
\item for all $\alpha$, $V_\alpha$ is an open subset of $\B(M)\times \TR^{\Sed(Y_\alpha)}$, where $M$ is a loopless matroid with $\rk(M)-1+\sed(Y_\alpha)=n$, and the map $\Phi_\alpha$ is a homeomorphism.
\item for all $\alpha, \alpha'$, \[\Phi_\alpha\circ\Phi_{\alpha'}^{-1}:\Phi_{\alpha'}(U_\alpha\cap U_\alpha')\longrightarrow \Phi_{\alpha}(U_\alpha\cap U_\alpha')\subset V_\alpha\] can be extended to an integral affine map $\varphi: \TR^{N_{\alpha'}}\rightarrow\TR^{N_{\alpha}}$.
\item the charts are of finite type, i.e.,\ there exists a finite number of open sets $(Q_i)$ such that $\bigcup Q_i=X$, and such that for every $Q_i$ there is an $\alpha$ such that $Q_i\subset U_\alpha$ and $\Phi_\alpha(Q_i) \subset V_\alpha$.
\end{compactenum}
We refer to the above varieties as ``abstract'', and reserve the term \emph{smooth tropical variety} only for those varieties 
that are embedded as polyhedral complexes in $\mathbb{A}^d$ (such that the $U_\alpha=V_\alpha$ are subsets of $\mathbb{A}^d$ and the maps are identities). Of special importance is the \emph{(tropical) hyperplane}, a hypersurface that arises as the smooth tropical variety of a single, tropical (affine) linear function, none of whose coefficients are degenerate. Combinatorially, a tropical hyperplane arises as the Bergman fan of a uniform matroid of rank $r$ on $r+1$ elements.

\begin{rem}[Necessity of the integral structure]\label{rem:real}
This is the standard definition of smooth tropical varieties following Mikhalkin, but for our purposes, the integrality assumption in (2) can be dropped. Indeed, the only case of a ``tropical Lefschetz Theorem'' that depends on the integrality assumption considered here, will turn out to fail regardless of the integral structure (Proposition~\ref{prp:pushing_direction}).
\end{rem}

\subsection{Tropical link and tangent fan.} 
To understand the notion of an embedded smooth tropical variety, it is useful to consider tangent fan. For instance, a smooth tropical variety in $\mathbb{R}^d$ is a polyhedral complex such that at every point the tangent fan is a Bergman fan. For tropical varieties in tropical affine or projective space the answer is not quite so simple.
We need a notion of tangent fan at infinity:

Let $X$ be a polyhedral complex in affine tropical space $\TR^d$, and let $\sigma$ denote a face of $X$ of fine sedentarity $S$. Then $\tT_\sigma X\defeq  \TT_\sigma \big(X\cap (\RR^{[d]{{\setminus}} S} \times \{-\infty\}^S)\big)$ is the \emph{tropical tangent fan} of $\sigma$ in $X$. Let $\tau$ be any facet of $\sigma \in X$ in $\TR^d$.
 Then there is a natural map \[\mr{d}_{\sigma\rightarrow\tau}: \tT_\sigma X \longrightarrow \tT_\tau X.\] If $\Sed(\sigma)=\Sed(\tau)$, then $\mr{d}_{\tau\rightarrow\sigma}$ is given by natural inclusion of tangent fans. If $\Sed(\sigma)\neq \Sed(\tau)$, then $\Sed(\sigma)\subsetneq \Sed(\tau)$ and $\mr{d}_{\tau\rightarrow\sigma}$ is given by restriction of the projection along coordinate directions
\[\RR^{[d]{{\setminus}} \Sed(\sigma)}\longrightarrow \RR^{[d]{{\setminus}} \Sed(\tau)}.\]

With this notion, we can rephrase again: 
\begin{prp}
An embedded smooth tropical variety in $\mathbb{A}^d$ is a polyhedral complex such that the tropical tangent fan at each point is a Bergman fan.
\end{prp}

\begin{proof}
It is clear that the notions coincide at mobile points. For points of positive sedentarity the notion of tropical tangent fan and condition (1) for smooth tropical varieties reduce the proof to checking the claim only at mobile points. This is done by restricting the variety to the appropriate subspace of positive sedentarity.
\end{proof}
%

\section{The positive side of the Bergman fan.}\label{ssec:pos} 

For the purpose of proving the Lefschetz theorems we need to first recast Theorem~\ref{mthm:pos_sum_geom_latt_cm} in the language of Bergman fans. For the tropical Lefschetz theorems, we are interested in the topology of the restriction
$\Lk_{\mbf{0}} \RS{\B}{{H}^+}\simeq \TT^1_{\mbf{0}} (\B\cap H^+)$ of a Bergman fan $\B$ to a closed halfspace ${{H}^+}$, cf.\ Section~\ref{ssec:restref}.
We have the following corollary of Theorem~\ref{mthm:pos_sum_geom_latt_cm}. 
\begin{mlem}\label{mlem:positive-side-bergman}
Let $M$ be a finite matroid, let $\B=\B(M)$ in $\RR^{|M|-1}$ be its Bergman fan, and let $H^+$ be a general position closed halfspace with $\mbf{0}\in\partial H^+$. Then the link $\Lk_{\mbf{0}} \RS{\B}{{H}^+}$ is homotopy Cohen--Macaulay of dimension $r-2$.
\end{mlem}

\begin{proof}
Let $\mbf{n}$ denote the interior normal vector to $H$. Then \[\omega=(\omega_1,\cdots,\omega_n)=(\mbf{n}\cdot e_1,\cdots,\mbf{n}\cdot e_n)\]
is a generic weight on the elements $[n]$ of $M$ with $\omega\cdot [n]=0$. 
With this we have, for every subset $\sigma$ of $[n]$, that
\[\sigma\in \LL^{>0}\Longleftrightarrow \omega\cdot \sigma>0\Longleftrightarrow\mbf{n}\cdot e_{\sigma}>0 \Longleftrightarrow e_{\sigma}\in {H}^+\]
so that 
\[\Lk_{\mbf{0}} \RS{\B}{{H}^+}\cong \LL^{>0}.\]
The claim hence follows from Theorem~\ref{mthm:pos_sum_geom_latt_cm}.
\end{proof}

Similarly, we also have the following topological characterization of the full Bergman fan using shellability \cite{Bj1},
or Theorem~\ref{mthm:pos_sum_geom_latt_cm} for $t << 0$, to conclude Cohen-Macaulayness for the whole lattice $\LL$.

\begin{mlem}\label{mlem:bergman}
Let $M$ denote a finite matroid and $\B=\B(M)$ in $\RR^{|M|-1}$
 its Bergman fan. Then the link $\Lk_{\mbf{0}} \B$ is homotopy Cohen--Macaulay of dimension $r-2$.
\end{mlem}

This fact is quite central, since together with Lemma~\ref{lem:balance} it implies that the toric variety over a Bergman fan has
a Gorenstein Chow ring, even though the fan is non-complete, which in turn is a central ingredient in~\cite{AHK}.

\section{Lefschetz Section Theorems for cell decompositions of tropical varieties.} \label{sec:cell}

We are now ready to prove several Lefschetz theorems for tropical varieties. In each section we first recall 
the classical Lefschetz theorems, and then proceed to prove analogues for tropical varieties. It should be noted, however,  that the tropical versions are often so much stronger than the classical "analogues" that we do not even know how to formulate the proper analogues classically. They should therefore be seen as inspirations rather than true analogues.

A crucial ingredient of the Lefschetz Section Theorem of Andreotti--Frankel \cite{AndreottiFrankel} is a vanishing theorem for CW models and Betti numbers of affine varieties. 

\begin{thm}[Andreotti--Frankel, \cite{AndreottiFrankel}]\label{thm:Andreotti}
Let $X$ denote a smooth $n$-dimensional variety in~$\mbb{C}^d$. Then $X$ is homotopy equivalent to a CW complex of dimension $\le n$.
In particular, the integral homology groups of $X$ vanish above dimension $n$.
\end{thm}

The idea for the proof is to use classical Morse theory; the Morse function is given by the distance $f$ from a generic point in $\mbb{C}^d$. The theorem then follows from the main lemma of Morse theory, together with a simple index estimate for the critical points of $f$. See Milnor's book \cite{Milnor} for an excellent exposition. 

From this affine theorem, one can deduce the classical Lefschetz Section Theorem. For smooth algebraic varieties and homology groups this theorem was proven first by Lefschetz \cite{Lefschetz}, and later by Andreotti--Frankel~\cite{AndreottiFrankel}.

\begin{thm}[Lefschetz, Andreotti--Frankel]
\label{thm:clef}
Let $X$ denote a smooth algebraic $n$-dimensional variety in $\mbb{CP}^d$, and let $H$ denote a hyperplane in $\mbb{CP}^d$. Then the inclusion of $X\cap H$ into $	X$ induces an isomorphism of integral homology groups up to dimension $n-2$, and a surjection in dimension $n-1$.
\end{thm}

This follows directly from the fact that $X{{\setminus}} H$ is an affine variety, Theorem~\ref{thm:Andreotti}, and Lefschetz duality. Bott, Thom and Milnor \cite{Bott, Milnor} then observed that this theorem extends to homotopy groups, and more generally to cell decompositions of the variety.

\begin{thm}[Bott, Milnor, Thom]\label{thm:clef2}
Let $X$ denote a smooth algebraic $n$-dimensional variety in $\mbb{CP}^d$, and let $H$ denote a hyperplane in $\mbb{CP}^d$. Then
$X$ is, up to homotopy equivalence, obtained from $X\cap H$ by successively attaching cells of dimension~$\ge n$. 

In particular, the inclusion $X\cap H \hookrightarrow X$ induces an isomorphism of homotopy and integral homology groups up to dimension $n-2$, and a surjection in dimension $n-1$.
\end{thm}

\subsection*{The tropical case.} Similarly to the Vanishing Theorem for classical projective varieties, the Lefschetz type theorem for affine tropical varieties proved in this section is crucial for deriving Lefschetz theorems for projective varieties. In the tropical realm, the Andreotti--Frankel Vanishing Theorem for affine varieties takes a form similar to the Lefschetz Theorem for projective varieties. The theorem can be stated as follows:

\begin{mthm}\label{mthm:lef_t_trop_aff}
Let $X\subset \mathbb{A}^d$ be a smooth $n$-dimensional tropical variety, and let $H$ denote a chamber complex. Then $X$ is, up to homotopy equivalence, obtained from $X\cap H$ by successively attaching $n$-dimensional cells.
\end{mthm}

By elementary cellular homology and homotopy theory \cite{Hatcher, Whitehead}, we immediately obtain a Lefschetz Section Theorem for homotopy and homology groups:

\begin{mcor}\label{mcor:leftbt}
Let $X\subset \mathbb{A}^d$ be a smooth $n$-dimensional tropical variety, and let $H$ denote a chamber complex. Then the inclusion $X\cap H \hookrightarrow X$ induces an isomorphism of homotopy groups resp.\ integral homology groups up to dimension $n-2$, and a surjection in dimension $n-1$.
\end{mcor}

\begin{rem}
Throughout, every tropical variety is considered to be endowed with a triangulation. This is of no further use than merely 
to allow us to analyse varieties using methods from combinatorial and PL topology. In particular, the triangulation at hand does not
need to be specified in our setting.
\end{rem}

\begin{mlem}\label{mlem:lef_to_convex_cell}
Let $X$ denote a smooth tropical $n$-dimensional variety in $\mathbb{A}^d$, and let $P$ denote any closed pointed convex $d$-polyhedron or $d$-polytope in $\mathbb{A}^d$. Then $X\cap P$ is obtained from $X\cap \partial_{|\mo} P$ by successively attaching cells of dimension $n$.
\end{mlem}
Here $\partial_{|\mo} P$ is the \emph{mobile part} of $\partial P$, i.e.,
\[\partial_{|\mo} P\defeq  (\partial P)_{|\mo} =\{x\in \partial P: \sed(x)=0 \}\]

The following simple lemma explains why we need to work with pointed polyhedra.

\begin{lem}\label{lem:functions}
Let $P$ be a convex pointed polyhedron in $\mathbb{R}^d$. Then there exists a function
$\wt{f}:{P} \longrightarrow \RR^{\ge 0}$ that 
\begin{compactenum}[(1)]
	\item is smooth in the interior of $P$, 
	\item has strictly convex superlevel sets $\wt{f}^{-1}[t,\infty)$ for every $t>0$, 
	\item is strictly monotone increasing to $\infty$ on every infinite ray in the interior of $P$, and	
	\item such that $\wt{f}^{-1}\{0\}=\partial P$.
\end{compactenum}
If $X$ is any polyhedral complex in $P$, then $\wt{f}$ can be chosen to restrict to a stratified Morse function on~$X$.
\end{lem}

Notice that the fact that $P$ is pointed is necessary for condition (3). Indeed, if $P$ contained a line, then in particular it contains a ray and the ray in the opposite direction, which contradicts the monotone increasing property on at least one of these rays.

\begin{proof}
Functions satisfying (1)-(4) can be constructed rather easily, for instance as a product of the affine linear functions defining $P$. Let us start with
\[\widehat{f}(x)\ \defeq\ \prod_{H} \mr{d}(H,x)\]
where $H$ varies over the defining hyperplanes of $P$ and $\mr{d}$ denotes the euclidean distance. This function clearly satisfies (1) and (4). Moreover, if $v$ is any nonzero vector in the recession cone of $P$, then by pointedness of $P$ we have
$\widehat{f}(x+v) > \widehat{f}(x)$ because the distance to at least one of the boundary hyperplanes increases strictly, proving (3). To see (2), note that if $P$ is defined by $k$ hyperplanes, then $\widehat{f}(x)^{\nicefrac{1}{k}}$ is concave on $P$.

Consider now a polyhedral complex $X$ in $P$. To obtain a stratified Morse function on $X$, we need to perturb the function slightly. Condition (1) guarantees smoothness, and condition (3) guarantees that the gradient flow is uniformly outwardly oriented at $\infty$. Finally condition (2) guarantees that on every stratum $\sigma^\circ$, the function $\wt{f}_{|\sigma^\circ}:\sigma^\circ \rightarrow \RR^{\ge 0}$ has at most one critical value (namely, a minimum), cf.\ Lemma~\ref{lem:smt_observation_on_level sets}. 
To guarantee that these points are non-degenerate and have distinct critical values, perturb to
\[\wt{f}(x)\ \defeq\ \prod_{H} \mr{d}^{\alpha_H}(H,x)\]
for generic choices of positive reals $\alpha_H$.
\end{proof}	

\begin{proof}[\textbf{Proof of Lemma~\ref{mlem:lef_to_convex_cell}}]
Note that in the commutative diagram 
\[\begin{tikzcd}
X_{|\mo} \cap P  \arrow[hookrightarrow]{r} & X \cap P\\
X_{|\mo} \cap \partial P \arrow[hookrightarrow]{r} \arrow[hookrightarrow]{u} &  X \cap \partial P \arrow[hookrightarrow]{u}
\end{tikzcd}
\]
the horizontal arrows induce isomorphisms in homotopy. We may therefore restrict to the mobile parts of $X$ and $P$ for the remainder of the proof, and will suppress this from notation for simplicity by omitting the subscript $_{|\mo}$.

Let $\wt{f}:{P} \longrightarrow \RR^{\ge 0}$ be chosen as in the previous lemma, so that we obtain a stratified Morse function for $X$. As already observed, all critical points are minima under this Morse function.


Let $X_{\le s}\defeq  X\cap f^{-1}(-\infty,s]$ and assume that $t$ is a critical value of $f$, with critical point $x$ in $\RR^d$. Let $\varepsilon>0$ be chosen small enough, so that $(t-\varepsilon,t]$ contains only one critical value of $f$. 
 
\begin{figure}[htb] 
\centering 
 \includegraphics[width=0.83\linewidth]{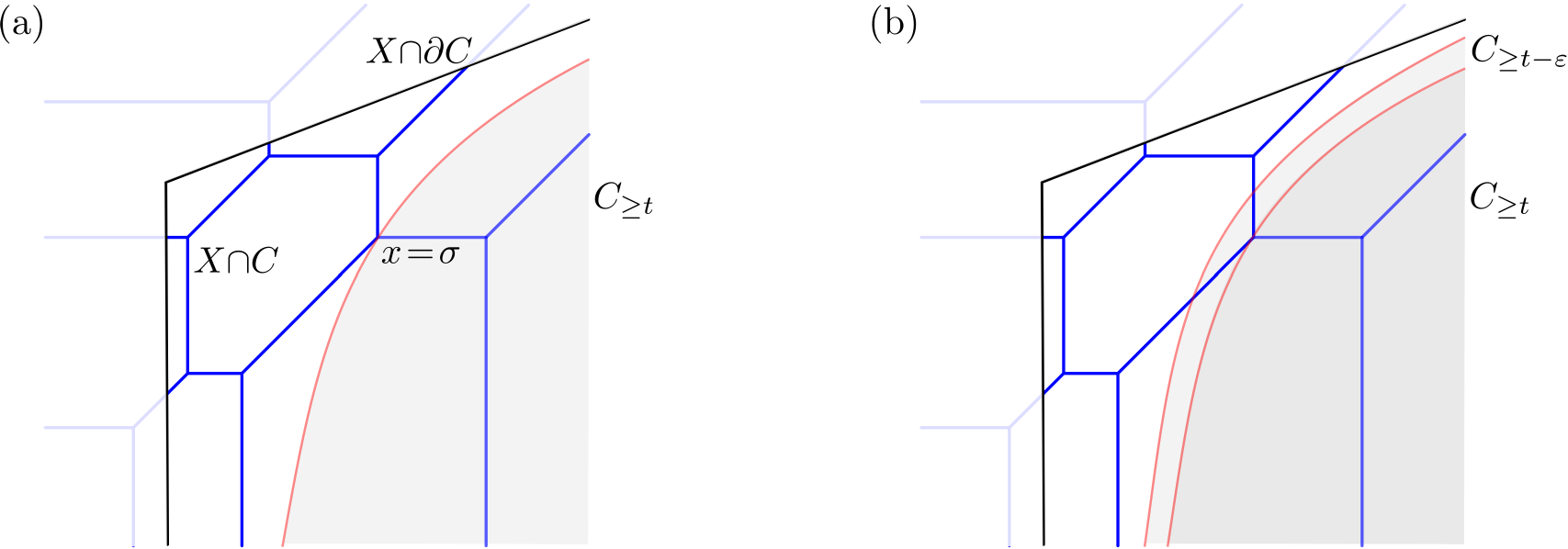} 
\caption{\small Using stratified Morse theory on $X\cap P$, it suffices to consider the Morse data at critical points.} 
 \label{fig:sublevel}
\end{figure}

Let $\sigma$ denote the minimal face of $X$ containing $x$. The set $P_{\ge t}\defeq  \wt{f}^{-1}[t,\infty)$ is a convex set with smooth boundary in $P$. By Lemma~\ref{lem:smt_observation_on_level sets}(1), the tangential Morse data at $x$ is therefore given by $(\sigma,\partial \sigma)$. If we now consider the halfspace $\TT_x P_{\le t}$ then we see that the normal Morse data at $x$ is given by $(\CO \RN^1_\sigma X_{\le t}, \RN^1_\sigma X_{\le t})$ where $\RN^1_\sigma X_{\le t}=\RN^1_\sigma X \cap \RN^1_\sigma f^{-1}(-\infty,t]$ is homotopy equivalent to a wedge of spheres of dimension $(n-\dim\sigma-1)$ by Lemma~\ref{mlem:positive-side-bergman}.

\begin{figure}[htb] 
\centering 
 \includegraphics[width=0.67\linewidth]{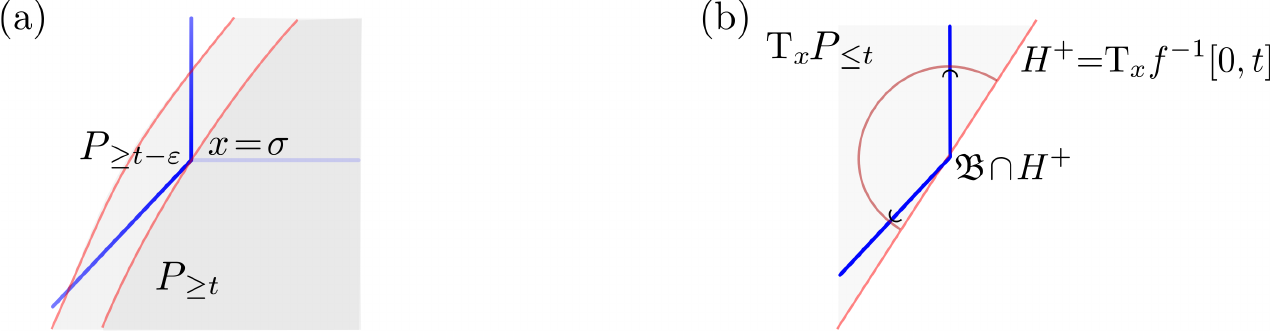} 
 \caption{\small The normal Morse data at a critical point $x\subset \sigma^\circ$ is given by restricting $\Lk_\sigma X\simeq \RN_\sigma^1$ to the hemisphere $\TT^1_x P_{\le t}$.} 
 \label{fig:Morse_data}
\end{figure}

Therefore, the Morse data at $x$ is given as 
\begin{equation*}\label{eq:Morsedata}
(\CO \RN^1_\sigma X_{\le t}, \RN^1_\sigma X_{\le t})\times (\sigma, \partial \sigma) \
\simeq \ \big(\CO ( \RN^1_\sigma X_{\le t}\ast \partial \sigma), \RN^1_\sigma X_{\le t}\ast \partial \sigma \big), 
\end{equation*}
where $\RN^1_\sigma X_{\le t}\ast \partial \sigma$ is homotopy equivalent to a wedge of $(n-1)$-spheres, by Lemma~\ref{lem:join}. The claim now follows from Theorem~\ref{thm:smt_main_theorem}(2), since $\partial P$ is the level set at $t=0$. This finishes the proof of Lemma~\ref{mlem:lef_to_convex_cell}.
\end{proof}

\begin{proof}[\textbf{Proof of Theorem ~\ref{mthm:lef_t_trop_aff}}]
Contrary to the classical case, all versions independent of the ambient space follow from a common lemma and do not use Lefschetz duality: By assumption, $H$ induces a partition of $\mathbb{A}^{d}$ into closed affine pointed polyhedra and polytopes $P_i$. Now, for every~$i$, we have by Lemma~\ref{mlem:lef_to_convex_cell} that $P_i\cap X$ is obtained from $X \cap \partial_{|\mo} P$ by attaching cells of dimension~$n$, as desired.
\end{proof}
	
\subsection{Stable intersection}\label{sec:stable} To prove the case of stable intersection, consider a generic perturbation $P_t \defeq P+vt$ of $P$, choosing $t$ small enough so that it encounters no critical points as $t>0$ tends to $0$, that is, so that $(\partial P_t) \cap X$ does not change combinatorial type. Consider  $P_t' \defeq P \cap P_t$. By Lemma~\ref{mlem:lef_to_convex_cell}, every $S^k$ in $P'_t \cap X$ can be homotoped to $\partial P_t \cap X$ for $k<n= \dim X$. In particular, we have a surjection of relative homotopy groups 
\[\pi_j (P_t  \cap X,(\partial P_t)  \cap X ) \ \twoheadrightarrow\ \pi_j (P_t'  \cap X,(\partial P_t')  \cap X )\]
for $j\le n$, that is an isomorphism for $j<n$. Hence, it remains to approximate $P$ by $P'_t$ from the interior. But this approximation leaves the relative homotopy groups invariant. Hence, we conclude that the relative homotopy groups of the pair $(X, X\cap_{\mathrm{stable}} H)$ vanish up to dimension $k< n$. We conclude.

\begin{cor}
Let $X\subset \mathbb{A}^d$ be a smooth $n$-dimensional tropical variety, and let $H$ denote a chamber complex in $\mathbb{A}^d$ for which the stable intersection is well-defined. Then $X$ is, up to homotopy equivalence, obtained from $X\cap_{\mathrm{stable}} H$, the limit of $X\cap H+vt$, by successively attaching $n$-dimensional cells. 

In particular, the inclusion $X\cap_{\mathrm{stable}} H \hookrightarrow X$ induces isomorphisms of homotopy groups resp.\ integral homology groups up to dimension $n-2$, and a surjection in dimension~$n-1$.
\end{cor}


%
%

\subsection*{Decomposing the variety, step by step.} It is possible to give a more ``combinatorial'' presentation of the proof of Lemma~\ref{mlem:lef_to_convex_cell} by exhibiting how the cells of a slightly refined version of $X$ are attached, one by one, along the sublevel sets of the Morse function. This is more in line with the Banchoff--K\"{u}hnel--Kuiper Morse theory for polyhedral complexes (cf.\ \cite{Banchoff, Kuehnel}) and Forman's combinatorial Morse theory. However, the stratified Morse theory used above works equally well. The key observation is as follows:


%

For $X$, $f$ and $P$ as in Lemma~\ref{mlem:lef_to_convex_cell}, let $\wt{X}$ denote the common refinement of $X$ and $P$, 
\[\wt{X}\defeq  X\cdot P=\{\sigma\cap \tau: \sigma \in X,\ \tau \in P \}\]
\begin{prp}\label{prp:geom_to_comb}
Let $(\wt{X}, f)$ be as above and let $t\ge 0$. Then \[\wt{X} \cap {f^{-1}(-\infty,t]}\ \simeq\ \RS{\wt{X}}{f^{-1}(-\infty,t]}.\] 
\end{prp}

\begin{proof}
Use Proposition~\ref{prp:restriction_gives_def_retract} and the convexity of superlevel sets of $\wt{f}$.
\end{proof}


%

%

\section{Lefschetz Section Theorems for complements of tropical varieties.}
\label{sec:compl}

Motivated by the study of complements of subspace arrangements, several Lefschetz theorems were proven that apply to complements of affine varieties, prominently the theorems of Hamm--L{\^e}, cf.\ \cite{DimcaPapadima, HammLe}.

\begin{thm}[Hamm--L\^e \cite{HammLe}, cf.\ \cite{DimcaPapadima, Randell}]
Let $\varphi$ denote a non-constant homogeneous polynomial in $d$ variables. If $H$ is a generic hyperplane in $\mbb{C}^d$, then \[C(\varphi)\defeq  \{x\in \mbb{C}^d: \varphi(x)\neq 0\}\] is, up to homotopy equivalence, obtained from $C(\varphi)\cap H$ by attaching cells of dimension $d$.
\end{thm}

\subsection*{The tropical case.} The purpose of this section is to provide a tropical analogue of this influential result. An \emph{almost totally sedentary hyperplane} is the closure in $\mathbb{A}^d$ of an affine hyperplane in~$\mathbb{R}^d$ as a polyhedron.

\begin{mthm}\label{mthm:left_comp_aff}
Let $X$ be a smooth $n$-dimensional tropical variety in $\mathbb{A}^d$, and let $C=C(X)$ denote the complement of $X$ in $\mathbb{A}^d$. Let furthermore $H$ be an almost totally sedentary hyperplane in $\mathbb{A}^d$, such that for any face $\sigma \in X$, $\aff \sigma$ intersects $H$ transversally or not at all. Then $C$ is, up to homotopy equivalence, obtained from $C\cap H$ by successively attaching $(d-n-1)$-dimensional cells. 

In particular, the inclusion of $C\cap H$ into $C$ induces an isomorphism of homotopy groups resp.\ integral homology groups up to dimension $d-n-3$, and a surjection in dimension $d-n-2$.
\end{mthm}

\begin{rmk}
More generally, in the situation of Theorem~\ref{mthm:left_comp_aff} one can take $H$ to be any closed polyhedron in $\mathbb{A}^d$ in sufficiently general position.
\end{rmk}

The central ingredient will be a relative version of Lemma~\ref{mlem:positive-side-bergman}.

\begin{lem}\label{mlem:negative-side-bergman}
Let $(M,\B,H^+,r)$ be as in Lemma~\ref{mlem:positive-side-bergman}, let $H^-\defeq \RR^{|M|-1}{{\setminus}} H^+$, and let $C=\RR^{|M|-1}{{\setminus}}\B$. Then $C$ is, up to homotopy equivalence, obtained from $C\cap H^-$ by attaching cells of dimension $|M|-r-1$.
\end{lem}

\newcommand{\PF}{\mathfrak{P}}

\begin{proof}
Let $\BB$ denote the Boolean lattice on the ground set of $M$, and let $\PF$ denote the associated Bergman fan. Finally, let $\om$ denote the weight associated to $H^+$ as given in the proof of Lemma~\ref{mlem:positive-side-bergman}, so that \[\Lk_{\mbf{0}} \RS{\B}{{H}^-}\cong \LL^{<0}.\]

Then $\RR^{|M|-1}{{\setminus}} \B \simeq \PF{{\setminus}} \B $ deformation retracts to $\BB-\LL$, and the retract restricts to a deformation retract of $H^-{{\setminus}} \B$ to \[\Lk_{\mbf{0}} (\B \cap H^-) \simeq\ \BB^{<0}- \LL^{<0}.\] 
Hence, the pair $(C,C\cap H^-)$ is homotopy equivalent to the pair $(\BB-\LL,\BB^{<0}-\LL^{<0})$. The claim follows by Corollary~\ref{cor:attaching_cells_complement}.
\end{proof}

We now need stratified Morse theory, not for polyhedra but for their complements. The main tools are very similar to those of Section~\ref{ssc:strat} and we refer the reader to \cite[Part~I]{GM-SMT} for details concerning this aspect of stratified Morse theory.

\begin{proof}[\textbf{Proof of Theorem~\ref{mthm:left_comp_aff}}.]
Perturb the distance $\mr{d}_{H}$ to $H$ to a Morse function $\wt{f}$ with convex sublevel sets. We now apply stratified Morse theory for complements of polyhedra: It suffices to prove that, if $x$ is any critical point of $f$, and $t$ its value, and $\varepsilon>0$ chosen small enough such that $[t,t+\epsilon)$ contains no further critical values of $f$, then $C_{\le t+\epsilon}=C\cap \wt{f}^{-1}[0,t+\varepsilon]$ is obtained from $C_{\le t}$ by successively attaching $(d-n-1)$-cells. 

Now, clearly the minimal stratum of $X$ containing $x$ is $x$ itself, so that the tangential Morse data at $x$ is trivial. It remains to estimate the normal Morse data at $x$. If we set \[H_x^-\defeq  \wt{f}^{-1} [t,\infty)\ \ \text{and}\ \ H_x\defeq  \wt{f}^{-1} \{t\} =\partial H_x^-,\] it is given by the relative link \[\big(\TT^1_x (X\cap H_x^-),\TT^1_x (X\cap H_x)\big).\] That is, $C_{\le t+\epsilon}$ is obtained from $C_{\le t}$ by attaching $\TT^1_x (X\cap H_x^-)$ along $\TT^1_x (X\cap H_x)$. Since by Lemma~\ref{mlem:negative-side-bergman}, $\TT^1_x (X\cap H_x^-)$ is obtained from $\TT^1_x (X\cap H_x)$ by successively attaching $(d-n-1)$-cells, the claim follows by the main lemma for stratified Morse functions.
\end{proof}

\section{Tropical {\it (p,q)}-homology.} The concept of $(p,q)$-homology was introduced by Mikhalkin~\cite{Shaw, IKMZ}. Despite its name, $(p,q)$-homology theory should be thought of as an analogue of Hodge theory in complex algebraic geometry. For more details, we refer the reader to \cite{IKMZ,MZEigenwave, Shaw, zbMATH06400709}. Here we discuss $(p,q)$-homology theory only over the reals, instead of tropical integral Hodge theory. For the integral case, see Section~\ref{ssc:intHodge}.

\subsection{\emph{p}-groups.} The coefficients of $(p,q)$-homology theory are given by the $p$-groups, which form analogues to the sheaf of differential forms in classical Hodge theory.

\newcommand{\FZ}{\mathbf{F}}

\begin{dfn}[$p$-groups] Let $\Sigma$ denote a polyhedral fan pointed at $v$, i.e.,\ any collection of rational polyhedral cones in $\RR^d$ pointed at $\mbf{0}$. For $p\ge 0$, we associate to $\Sigma$ the subgroup $(\FF_p\Sigma)_{|v}$ of $\bigwedge^p \mathbb{R}^d$ generated by elements $v_1\wedge v_2\wedge \cdots \wedge v_p$, where $v_1,v_2,\cdots,v_p$ are real vectors that lie in a common subspace $\lin{\sigma}, \sigma\in \Sigma$.
The groups $(\FF_p\Sigma)_{|v}$ are called the \emph{$p$-groups}.
\end{dfn}

\subsection{Homology from {\it p}-groups.} Let $X$ denote a smooth tropical variety (embedded or abstract), or more generally any polyhedral complex in tropical space. If $y$ is a point of $X$, and $p$ is a nonnegative integer, then we let $(\FF_p X)_{|y}$  denote the $p$-group of $\tT_y X$ (as defined in Section~\ref{sec:tangent}). 

Notice that if $\sigma$ is the minimal face containing $y$, and $z$ is any point contained in the relative interior of a face $\tau$ of $\sigma$, then we have a natural map \[(\FF_p X)_{|y}\ \longrightarrow\ (\FF_p X)_{|z}\]
induced by the map $\mr{d}_{\sigma\rightarrow\tau}$ defined above. This provides a local system of coefficients for a singular chain complex $C_\bullet(X;\FF_p X)$, the complex of \emph{(singular) $(p,q)$-chains}. 
The associated homology groups are the \emph{(singular) $(p,q)$-groups}.

\subsection{Some useful facts.} The $(p,q)$-groups are natural analogues of the classical Hodge groups in algebraic geometry: in \cite{IKMZ} it is proven that for $X$ a smooth tropical variety obtained as the limit of a one-parameter family $(X_t)$ of smooth complex projective varieties, the Hodge groups of a generic fiber $X_t$ are closely related to the $(p,q)$-groups of $X$. 
Moreover, it follows from classical arguments that, for realizable varieties $X$ (i.e.,\ varieties arising from classical varieties via tropicalization), we have the conjugation symmetry $H_q (X;\FF_p X)\cong H_p (X;\FF_q X)$, cf.\ \cite[Proposition~7.6]{MZEigenwave}. Not everything is analogous to the classical situation though: the $(p,q)$-homologies do not seem to satisfy the naive analogue of the Hodge Index Theorem \cite[Theorem~3.3.5]{Shaw}, and the extent to which positivity plays a role in the study of $(p,q)$-homology is not clear. 

We close this section by mentioning some useful results to keep in mind.

\begin{lem}[cf.\ {\cite[Proposition~5]{MZEigenwave}}] The $(p,q)$-groups are independent of the cell structure of the tropical variety chosen. \end{lem}

\begin{lem}\label{lem:link}
Let $X$ denote a polyhedral fan with conepoint $v$, and let $P$ denote a convex polytope containing $v$ in its interior. Then we have a natural map of chain complexes
\[C_{q-1}(X\cap \partial P ; \FF_p X)\longrightarrow C_q(X \cap P , X \cap \partial P ; \FF_p X)\]
that induces an isomorphism of $(p,q)$-groups.
\end{lem}

Here, the coefficient system on $C_q(X \cap P , X \cap \partial P ; \FF_p X)$ is defined as the pullback along the inclusion map $X\cap \partial P \hookrightarrow X$.


\begin{proof}
The desired map of chains is given by the join operation, which sends
$\sigma\in |X|\cap \partial P$ to the cell $v\ast \sigma$.\end{proof}


\subsection{Transversal chains}\label{ssec:transv}
According to a result of Mikhalkin and Zharkov, every homology class can be assumed to arise from a representative in general position.
\begin{dfn} 
Let $X$ be a smooth tropical $n$-dimensional variety, and let $\Sigma$ denote a decomposition of $X$ into polyhedra. We call a (singular) $(p,q)$-chain $c$ \emph{transversal} with respect to\ $\Sigma$ if 
every cell in the $k$-skeleton of the support of $c$ does not intersect the $\ell$-skeleton of $\Sigma$, provided that $k+\ell \le n-1$. 
We call a chain in $X$ \emph{transversal} if it is transversal with respect to some triangulation of $X$.
 \end{dfn}

\begin{lem}[cf.\ {\cite[Lemma~6.8]{MZEigenwave}}]\label{lem:push_it}
Let $X$ denote an $n$-dimensional smooth tropical variety. Let $c\in C_q(X;\FF_p X)$ denote a $(p,q)$-chain of $X$ such that for some face $\sigma$ of $X$ we have $\partial c\cap \rint \St_\sigma {X}=\emptyset$. If $q\le n-1$, then there is a $(p,q)$-chain $\wt{c}\in C_q(X;\FF_p X)$ which
\begin{compactenum}[\rm (1)]
\item is transversal with respect to $X$ when restricted to the open star of $\sigma$ in $X$,
\item is isomorphic to $c$ outside of $\St_\sigma^\circ X$, and such that
\item $(c-\wt{c})$ is the boundary of a $(p,q+1)$-chain of $X$ supported in $\St_\sigma X$.
\end{compactenum}
\end{lem}

Mikhalkin and Zharkov proved that this ``Pushing Lemma'', concerning pushing a chain away from a low-dimensional face, holds also with respect to integral $(p,q)$-homology theory, in contrast to the ``directional Pushing Lemma'' 
proved here later.

\begin{figure}[htb]
\centering 
 \includegraphics[width=0.78\linewidth]{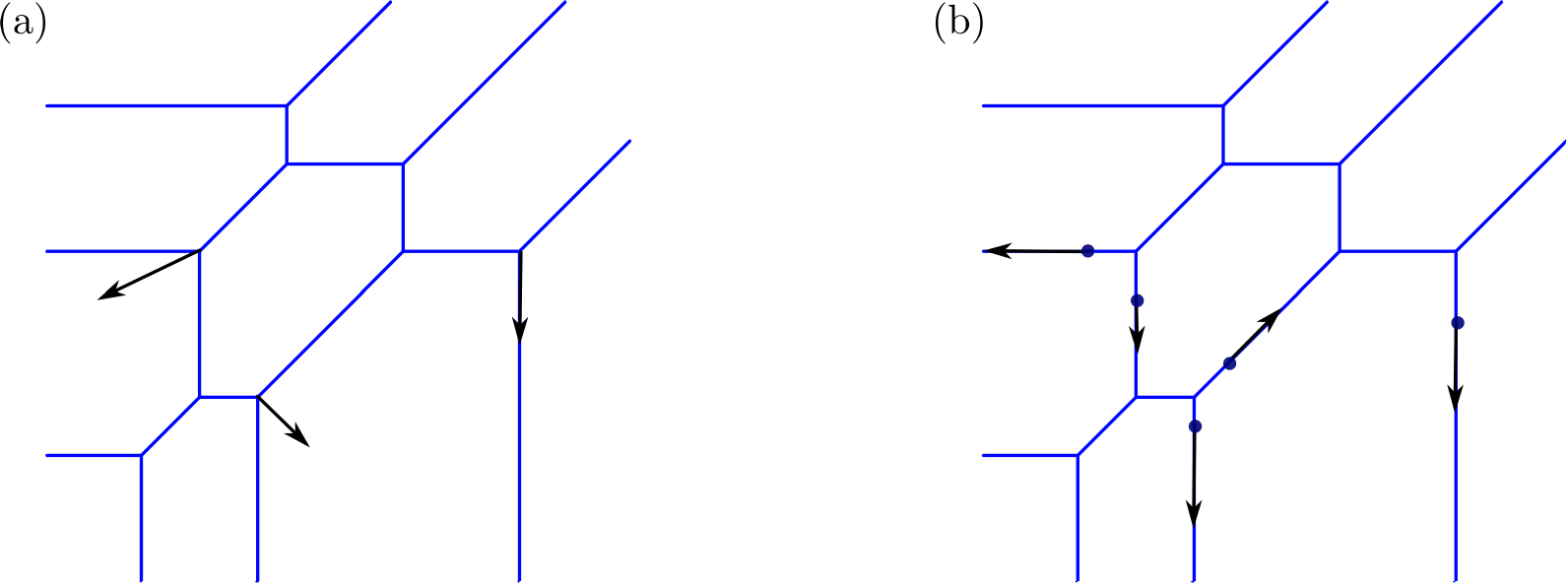} 
 \caption{\small Pushing a chain to transversality.} 
 \label{fig:stabilization}
\end{figure}

Let us note two corollaries: The first is that every homology class can be assumed transversal.

\begin{cor}[{\cite[Corollary~6.11]{MZEigenwave}}]\label{lem:stable}
Every $(p,q)$-homology class in a smooth tropical variety $X$ is represented by a transversal cycle. Moreover, a $(p,q)$-chain $\gamma$ in $X$ with transversal boundary is homologous to a transversal chain.
\end{cor}

The second consequence is that Lemma~\ref{lem:push_it}, applied to $(0,q)$-chains, gives a new proof of
Lemma~\ref{mlem:bergman}, at least for the special case of integral homology. We therefore recover the classical result of 
Folkman \cite{Folkman} on the homology of unfiltered geometric lattices.

\begin{cor}
The geometric lattice of a rank $r$ matroid is Cohen--Macaulay of dimension $r-2$.
\end{cor}

\section{Lefschetz Section and Vanishing Theorems for $(p,q)$-groups.} \label{sec:hodge}
The Lefschetz Section Theorem for Hodge groups was first established by Kodaira and Spencer.

\begin{thm}[Kodaira--Spencer \cite{KS}]\label{thm:ks}
Let $X$ be any smooth projective algebraic $n$-dimensional variety in $\mbb{CP}^d$, and let $H$ denote a hyperplane in $\mbb{CP}^d$. Then the inclusion $X\cap H \hookrightarrow X$ induces a map 
\[H^q(X,\Omega^p_X)\rightarrow H^q(X\cap H,\Omega^p_{X\cap H})\]
that is an isomorphism provided $p+q\le n-2$, and an injection for $p+q=n-1$.
\end{thm}

For a standard proof of this result, recall that by the Hodge Decomposition Theorem, we have
\[H^k(X,\mbb{C})\ =\ \bigoplus_{p+q=k} H^{p}(X,\Omega^q_X).\]
Together with the Dolbeault operators, this decomposition is functorial; the result now follows from Theorem~\ref{thm:clef} for complex coefficients. 
Alternatively, one can prove the theorem directly and algebraically, using the Vanishing Theorem of Akizuki--Kodaira--Nakano:

\begin{thm}[Akizuki--Kodaira--Nakano Vanishing Theorem]
Let $X\in \mbb{CP}^d$ denote a smooth, compact projective $n$-dimensional variety, and let $L\rightarrow X$ be a positive line bundle. Then 
\[H^q(X,\Omega^p_L)=0\ \text{for all}\ p+q>n.\]
\end{thm}

Theorem~\ref{thm:ks} then follows from the long exact sequence of Hodge groups and Serre duality, cf.~\cite{Voisin}. 

\begin{rem}
Notice that the logic of the first proof of the Kodaira--Spencer Lefschetz Theorem 
can also be reversed, and we can conclude Theorem~\ref{thm:clef} for complex coefficients from it. 
In other words, for smooth complex algebraic varieties, the Lefschetz theorem for Hodge groups is \emph{weaker} than the Lefschetz theorems of Lefschetz, Andreotti--Frankel and Bott--Milnor--Thom.
\end{rem}

\subsection*{The tropical case.} Contrary to the classical case, the analogous theorems for smooth tropical varieties are not as easily derived. Our Lefschetz Section Theorem for $(p,q)$-groups is stated as follows:

\begin{mthm}\label{mthm:proj_t_hodge_g}
Let $X$ denote an $n$-dimensional smooth tropical variety in $\mathbb{A}^d$, and let $H\subset \mathbb{A}^d$ denote an ample {{chamber complex}}. Then the inclusion $X\cap H\hookrightarrow X$ induces an isomorphism of $(p,q)$-homology
\[H_q(X\cap H; \FF_p (X \cap H))\ \longrightarrow\ H_q(X; \FF_p X) \]
for $p+q\le n-2$, and a surjection 
when $p+q=n-1$.
If $H$ has positive balancing (so that the stable intersection is well-defined), then we moreover have
\[H_q(X\cap_{\mathrm{stable}} H; \FF_p (X \cap_{\mathrm{stable}} H))\ \longrightarrow\ H_q(X; \FF_p X).\]
\end{mthm}

The analogous theorem fails for integral $(p,q)$-groups, see Section~\ref{ssc:intHodge}. For the proof of Theorem~\ref{mthm:proj_t_hodge_g}, 
given at the end of this section, we follow the classical, direct proof of the Kodaira--Spencer Lefschetz Section Theorem. 
That means, we first prove a tropical analogue
\ref{mthm:AKN_trop} of the Akizuki--Kodaira--Nakano Vanishing Theorem; the tropical Lefschetz Section Theorem for $(p,q)$-groups~\ref{mthm:proj_t_hodge_g} then swiftly follows.

\begin{mthm}\label{mthm:AKN_trop}
Let $X$ be an $n$-dimensional smooth tropical variety in $\mathbb{A}^d$, and let $P$ denote an ample pointed polyhedron  of codimension $k$ such that for every face $\sigma$ of $X$, $\mathrm{aff}(\sigma)$ is transversal to $\mathrm{aff}(P)$. Then
\[H_q(X\cap P, X\cap \partial_{|\mo} P; \FF_p X)=0\ \ \text{for all}\ \ p+q\le n-k-1.\]
\end{mthm}
The proofs of Theorems ~\ref{mthm:proj_t_hodge_g} and~\ref{mthm:AKN_trop}
 will follow after a sequence of lemmas. 

\subsection*{Pushing Chains, smooth tropical halflinks and the tropical AKN Theorem~\ref{mthm:AKN_trop}.}
The idea for the proof is to {``push''} $(p,q)$-chains in $X\cap P$ towards $X\cap \partial_{|\mo} P$. This in particular gives us a procedural view on the deformation of chains, and quickly implies the tropical AKN Theorem~\ref{mthm:AKN_trop} as in Lemma~\ref{mlem:lef_to_convex_cell}. 
 
\begin{mlem}[Pushing chains]\label{mlem:push_it}
Consider an $n$-dimensional smooth tropical variety $X$ in $\mathbb{A}^d$. Let $c\in C_q(X;\FF_p X)$ denote a $(p,q)$-chain of $X$ such that for some vertex $v$ of $X$, ${c}_{|\St_v X}$ is supported in a tropical geometric half-star $\St_v \RS{X}{H^+}$ (where $H^+$ is a closed halfspace in general position). 

Assume that $\partial c$ is not supported in $v$. 
If additionally $p+q\le n-1$, then there is a $(p,q)$-chain $\wt{c}\in C_q(X;\FF_p X)$ that
\begin{compactenum}[\rm (1)]
\item is supported in $X {\setminus} \{v\}$, 
\item is isomorphic to $c$ outside of a neighborhood of $v$ in $X$, and such that
\item $(c-\wt{c})$ is the boundary of a $(p,q+1)$-chain of $X$ supported in $\St_v X$.
\end{compactenum}
\end{mlem}

This lemma generalizes the corresponding Corollary~\ref{lem:stable} of Mikhalkin--Zharkov, where chains are deformed using a similar principle, but without forcing them into a certain direction. As we will see, this last aspect works only over the reals, compare also Proposition~\ref{prp:pushing_direction}.

The main ingredient of this ``Pushing Lemma'' will be a version of Lemma~\ref{mlem:positive-side-bergman}, which we reprove for $(p,q)$-groups, and the following lemma which establishes that framings (i.e.,\ coefficients of a $(p,q)$-chain) can be pushed past critical strata. For this, we prove a Lefschetz theorem for $\FF_p$-groups, similar to Theorem~\ref{mthm:pos_sum_geom_latt_cm}. For a Bergman fan $\B$ associated to a geometric lattice $\LL$, and $\om$ a weight on the atoms of $\LL$, let $\B^{> t}$ denote the restriction of the Bergman fan to the chains of flats with weight $>t$.

\begin{lem}\label{lem:framing_half}
Consider a Bergman fan $\B$ of dimension $n$ in $\R^d$ and let $H^+$ denote a closed general position halfspace whose boundary contains the origin.

 Then $(\FF_p \B)_{|\mathbf{0}}$ is generated by $(\FF_p (\B\cap H^+))_{|\mathbf{0}}$ for $p\le n-1$.\end{lem}

We shall prove this fact in higher generality and in three steps. Recall that a polyhedral complex is pure if all its facets are of the same dimension.

We call a polyhedral fan $\Fan$ \emph{nowhere acyclic} if it is pure and, for any point $p$ of $\Fan$, the tangent fan to $\Fan$ in $p$ does not lie in a closed halfspace with linear boundary unless it lies within that boundary. 

We shall need three lemmata. We start with the simplest:

\begin{lem}
Bergman fans are nowhere acyclic.
\end{lem}

\begin{proof}
This is immediate from Lemma~\ref{lem:balance}: If $p$ lies in the interior of a facet of $\B$, then the tangent fan is a subspace. If $p$ does not lie in the relative interior of a facet, then it lies in a face of codimension one, where the desired follows from the balancing property.
\end{proof}

The second lemma is more challenging.

\begin{lem}\label{lem:framing_halff}
Consider a nowhere acyclic fan $\Fan$ of dimension $n\ge 2$ in $\R^d$ and let $H^+$ denote a closed general position halfspace whose boundary contains the origin. 

Then $(\FF_1 \Fan)_{|\mathbf{0}}$ is generated by $(\FF_1 (\Fan\cap H^+))_{|\mathbf{0}}$.\end{lem}

\begin{proof}
We may assume, without loss of generality, that $(\FF_1 \Fan)_{|\mathbf{0}} \cong \R^d$. It remains to prove that $(\FF_1 (\Fan \cap H^+))_{|\mathbf{0}}$ spans $\R^d$. As $H^+$ is assumed to be in general position (no ray of $\Fan$ lies in its boundary) the conclusion of the lemma is invariant under small perturbations, so we may assume that the position of $H^+$ is sufficiently generic. 

A \emph{subhalfspace} is the intersection of a (linear) hyperplane with a closed halfspace whose boundary is a linear hyperplane transversal to~the first hyperplane. A subhalfspace $g$ is \emph{in the same pencil} as a sub\-halfspace $h$ if \[h\cap g=\partial h=\partial g.\]

We need to prove that $\Fan \cap H^+$ is not contained in a subhalfspace $h^+\subsetneq H^+$. Assume the contrary, i.e.\
assume that 
\begin{equation}
\label{eq:ass}
\text{$\Fan \cap H^+$ is contained in $h^+\subsetneq H^+$}.
\end{equation}
For subhalfspaces $g$ in the same pencil as $h^+$, define $\alpha(g)$ as the angular distance of $g$ to $h^+$, see Figure~\ref{fig:degeneration}. 
\begin{figure}[htb] 
\centering 
 \includegraphics[width=0.45\linewidth]{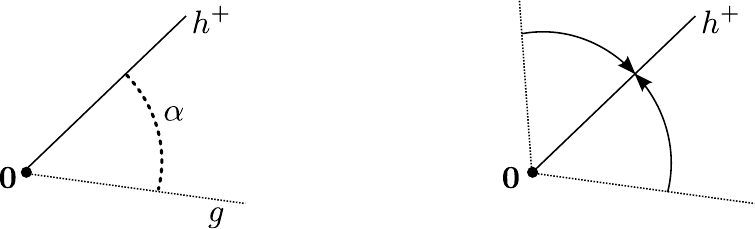} 
 \caption{\small For a geometric halflink, we study the angle to a subhalfspace $h^+$.} 
 \label{fig:degeneration}
\end{figure}

\noindent Set \[\alpha_0\ \defeq \ \sup\{t: \alpha^{-1} [0,t] \cap \Fan\ =\ h^+\cap \Fan\ =\ H^+\cap \Fan\}.\]
Then we have $\alpha_0>0$ (since we assumed $h^+\cap \Fan = H^+\cap \Fan$), and 
since $\Fan$ spans $\R^d$ we also have $\alpha_0<\pi$.

\begin{figure}[htb] 
\centering 
 \includegraphics[width=0.45\linewidth]{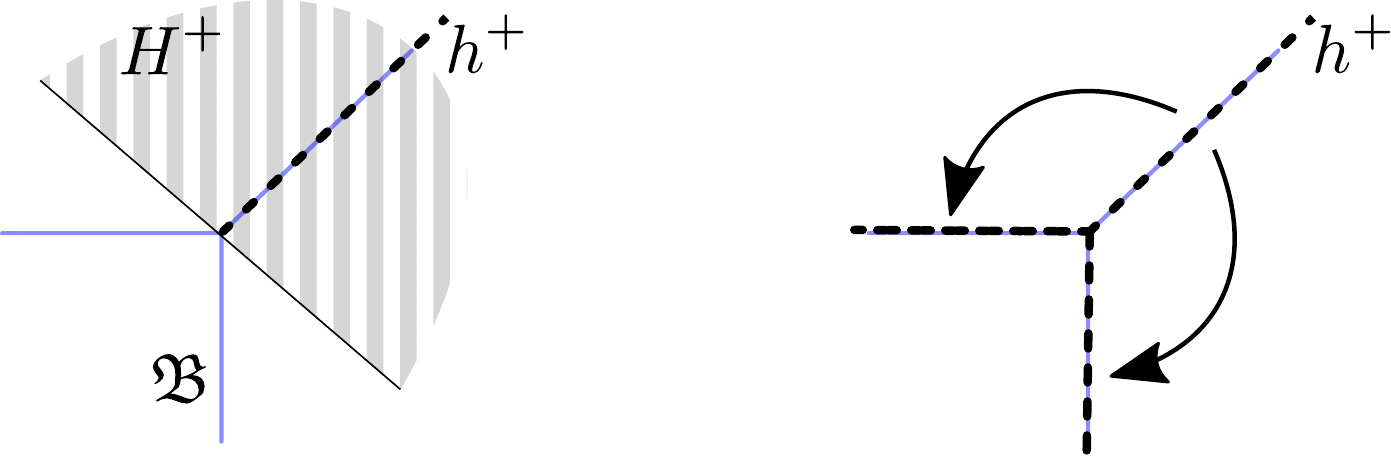} 
 \caption{\small If $\Fan$ spans the embedding vector space, but $\Fan \cap H^+$ lies in a subhalfspace $h^+$, then we may move through subhalfspaces with increasing angle to $h^+$ until, at some angle strictly between $0$ and ${\pi}$, the subhalfspace intersects $\Fan$ in a new point for the first time.} 
 \label{fig:monodromy}
\end{figure}

Since $\alpha_0$ is chosen as supremum, one of the two subhalfspaces $g_0$ at distance $\alpha_0$ from $h^+$ contains a point $p$ of $\Fan$ in its interior. By genericity of $H^+$ no facet of $\Fan$ intersects $g_0$ in an $n$-dimensional set.

\begin{figure}[htb] 
\centering 
 \includegraphics[width=0.35\linewidth]{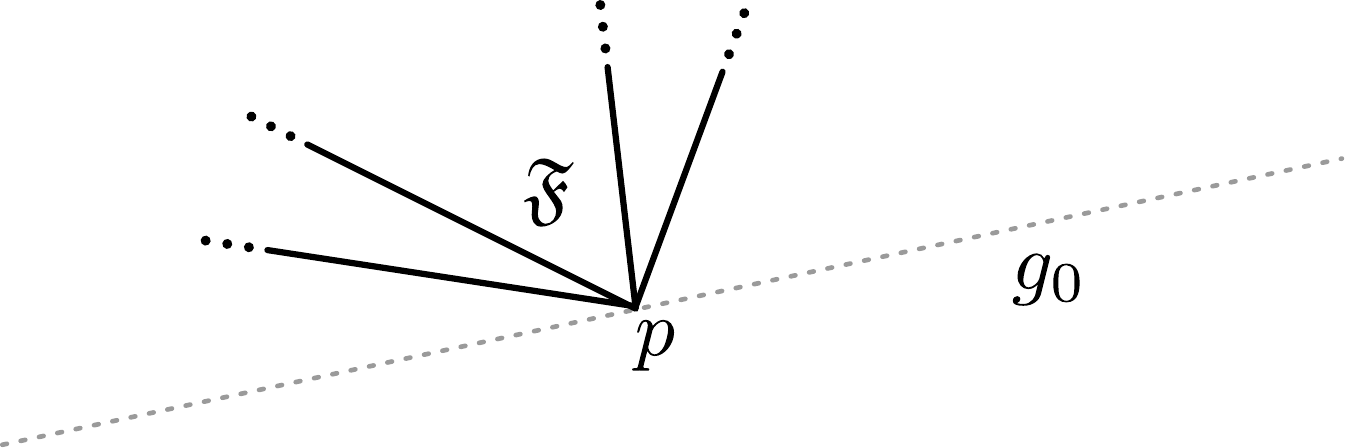} 
 \caption{\small Lemma~\ref{lem:framing_half}, $p=1\le n-1$, follows from the balancing property of Bergman fans.}
 \label{fig:push_framing_local}
\end{figure}

By construction, \[\alpha^{-1}{[0,\alpha_0)} \cap \Fan \ =\ h^+\cap \Fan\ =\ H^+\cap \Fan.\] Hence all cones of $\Fan$ incident to $p$ lie in a closed halfspace with boundary point $p$ (compare Figure~\ref{fig:push_framing_local}). In particular, at least one of these cones, specifically any facet of $\Fan$ incident to~$p$, intersects the interior of said halfspace, as none of the $n$-dimensional cones intersects the boundary of that halfspace in an $n$-dimensional set. But this contradicts the nowhere acyclic property of~$\Fan$.
\end{proof}

Finally, we have the following lemma, which together with the two preceding ones gives Lemma~\ref{lem:framing_half}.

\begin{lem}\label{lem:framing_halfff}
Consider an nowhere acyclic fan $\Fan$ of dimension $n\ge 2$ in $\R^d$ and let $H^+$ denote a closed general position halfspace whose boundary contains the origin. 

Then $(\FF_p \Fan)_{|\mathbf{0}}$ is generated by $(\FF_p (\Fan\cap H^+))_{|\mathbf{0}}$ for all $p\le n-1$.\end{lem}

\begin{proof}
We argue by induction on $p$. Note for this that a nowhere acyclic fan has the property that the normal fan at every face is nowhere acyclic.
We already proved the case $p=1$, so we may assume that $p>1$.
Let $f$ denote any generic continuous function $\R^d\setminus \{\mathbf{0}\} \rightarrow \R$ whose superlevel sets $f^{-1}[t,\infty)\cup \{\mathbf{0}\}$, $t> 0$, are convex cones pointed at the origin and are otherwise smooth.
Assume further that $f$ is chosen so that $f^{-1}[0,\infty)=\overline{H{\setminus} H^+}$. Specifically, we can simply choose $f$ to be given by the angular distance to the interior normal of $H^+$, minus $\frac{\pi}{2}$.

For the proof, we consider a representative $\gamma$ of an element in $(\FF_p \Fan)_{|\mathbf{0}}$. We may assume that $\gamma$ is a weighted sum of $p$-dimensional cones of $\Fan$, that is, 
\begin{equation}\label{eq:red}
\gamma\ =\ \sum_{\sigma \text{ $p$-face of $\Fan$}} a_\sigma \sigma
\end{equation}
where $a_\sigma\in \R$ and we identify $\sigma$ with the exterior product of vertices along its rays. We say that $\gamma$ is \emph{$t$-reduced}, and write $\gamma_{\le t}$, if every $p$-dimensional face occuring on the right hand side of \eqref{eq:red} with a non-zero coefficient has a ray $\rho$ with $f(\rho)\le t$.

Our goal is to prove that
\[\gamma\ =\ \gamma_{\le 0}\ \ \text{in}\ \ \bigwedge^p \R^d,\]
that is, as we decrease $t$, we push $\gamma$ to an equivalent element that is $0$-reduced. As we shall see, this in particular implies that $\gamma$ is generated by $(\FF_p (\Fan\cap H^+))_{|\mathbf{0}}$.

\begin{figure}[htb] 
\centering 
 \includegraphics[width=0.8\linewidth]{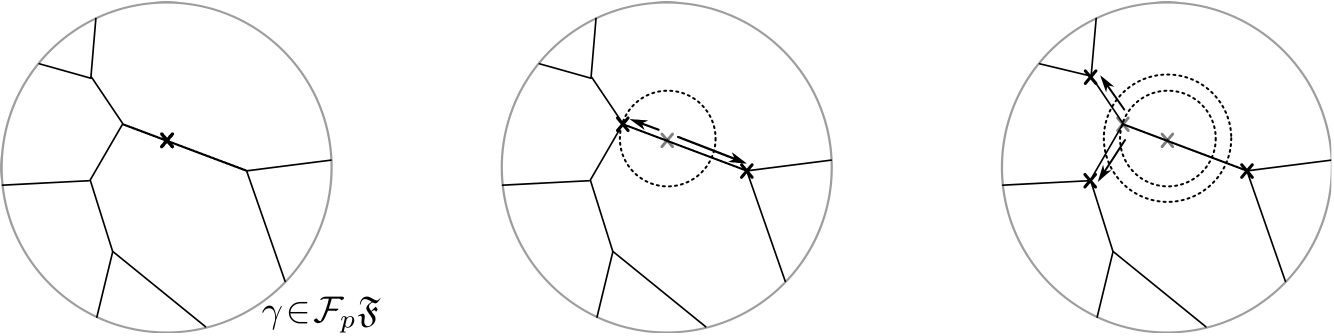} 
 \caption{\small We push the generators of a framing $\gamma$ to an equivalent framing on the positive side of a Bergman fan along a Morse function.}
 \label{fig:push_coefficients}
\end{figure}

For the former, we see that we have to be careful to analyze the points in time when the level sets of $f$ pass a ray of $\Fan$, say at time $t_0$. At such a point, one of the summands of $\gamma$ may see all its rays in $f^{-1}[t_0,\infty)$. But then we can represent that summand by elements in $f^{-1}(-\infty,t_0]$ by induction on $p$. Repeating this as $f$ passes the rays of $\Fan$ yields the desired representation $\gamma_{\le 0}$ of $\gamma$.

 Finally, it is clear that $\gamma_{\le 0}$ can be presented using $(\FF_p (\Fan\cap H^+))_{|\mathbf{0}}$, as by genericity of $H^+$ each of the summands has a vertex in the interior of $H^+$.
\end{proof}

We now prove a $(p,q)$-analogue of Theorem~\ref{thm:filtgl}.

\begin{lem}\label{lem:intervals_pq}
Let $\LL$ denote the lattice of flats of a matroid on the ground set $[m]$ and of rank $r\ge 2$, let $\B$ denote its Bergman fan, and let $\omega$ denote any generic weight on its atoms. Let $t$ denote any real number with $t\le \min\{0,\omega\cdot [m]\}$. Then, for any polytope $P$ containing the origin in its interior, we have 
\[H_q(\B^{>t}\cap P, \B^{>t}\cap \partial P;\FF_p \B)\ =\ 0 \text{ for }p+q\le r-2.\]
\end{lem}

Here, to define $\B^{> t}\defeq \zeta(\Delta(\LL^{> t}))$ we use the natural combinatorial isomorphism \[\Delta(\LL)\ \xrightarrow{\ \zeta\ }\  \B/\{\varnothing\}\] that sends a chain $\mbf{C}$ in $\LL$ to the cone $\pos(\mbf{C})$. 

\begin{proof}
By Lemma~\ref{lem:link}, it suffices to prove that $H_q(\B^{>t}{\setminus} \{\mbf{0}\};\FF_p \B)=0$ for $p+q\le r-3$.

Consider first the permutahedral fan $\PF$ (i.e.\ the Bergman fan over the Boolean lattice $\BB=\BB[m]$): By Theorem~\ref{thm:bc}, the associated part of the Bergman fan $\PF^{>t}$ is the cone over an $(m-2)$-dimensional Cohen--Macaulay complex in $\R^{m-1}$. Therefore, we see that the $(p,q)$-groups $H_q(\PF^{>t}{\setminus} \{\mbf{0}\};\FF_p \PF^{>t})$ vanish for $p+q\le m-3\ge r-3$.

We canonically extend $\zeta$ to a map $\Delta(\BB)\longrightarrow \PF$. In abuse of notation, we further extend $\zeta$ to a map of posets by defining $\zeta(\mc{P})\defeq \zeta(\Delta(\mc{P}))$ for a poset $\mc{P}\subset \BB$. 
Since
\[H_q(\PF^{>t}{\setminus} \{\mbf{0}\};\FF_p \PF^{>t})\ \cong\ H_q(\PF^{>t}{\setminus} \{\mbf{0}\};\FF_p \zeta (\BB^{>t}\cup \LL)),\] the latter vanish as well under the same conditions.



Now, we argue as in Theorem~\ref{thm:filtgl} and Lemma~\ref{mlem:quillen}: To transition from $\PF^{>t}$ (which we already understand) to $\B^{>t}$,
we remove the rays of $\PF^{>t}$ not in $\B^{>t}$ one by one according to decreasing rank.

Consider an intermediate fan $\B^{>t}\subset\mathfrak{I}\subset \PF^{>t}$, and ${\rho}$ a ray in $\mathfrak{I}- \B^{>t}$ corresponding to a maximal element $\mu\defeq \zeta^{-1}({\rho})$ in the intermediate poset $\mc{I}\defeq \zeta^{-1}(\mathfrak{I}-\B^{>t}) \subset \BB$. By Lemma~\ref{lem:Bergmanp}, the normal fan of a ray in a Bergman fan is again a Bergman fan, so we can use induction on the rank and obtain 
\[H_{q-1}(\mathfrak{I}{\setminus} (\zeta(\mc{I}{\setminus}\{\mu\})\cup {\rho}); \FF_p \zeta(\mc{I} \cup \LL))\ =\ 0\]
for all $p+q\le r-2$. Hence, Lemma~\ref{lem:intervals_pq} follows with Lemma~\ref{lem:link} and by a second induction on the number of elements in $\BB^{>t}{\setminus} \LL^{>t}$.
\end{proof}

\begin{rem}
Instead using the fact that Bergman fans locally look like Bergman fans, we could have worked using a more general principle, based on the following analogue of Lemma~\ref{lem:join}.
\begin{prp}\label{lem:Kn}
Let $X_1, X_2$ denote two fans, and assume that the $(p,q)$-groups of $X_i{\setminus} \{\mbf{0}\}$ vanish for $p+q\le a_i$ $(i=1,2)$.
Then the $(p,q)$-homology of $(X_1\times X_2){\setminus} \{\mbf{0}\}$ vanishes for $p+q\le a_1+a_2+2$.
\end{prp}

For the proof, observe that $(X_1\times X_2){\setminus} \{\mbf{0}\}$ is obtained as the join of $X_1{\setminus} \{\mbf{0}\}$ and $X_2{\setminus} \{\mbf{0}\}$. Following the proof of Lemma~\ref{lem:join}, it remains to prove a K\"unneth theorem for tropical homology, which states that for tropical varieties $X,\, Y$ in euclidean vectorspaces, we have
\[H_q(X\times Y;\FF_p (X\times Y))\ \cong\ \bigoplus_{\substack{q_X+q_Y=q \\ p_X+p_Y=p}} H_{q_X}(X;\FF_{p_X} X) \otimes
H_{q_Y}(Y;\FF_{p_Y} Y).\]
This in turn is proven as the classical K\"unneth theorem, by first restricting to the case 
\[H_q(X\times (P,\partial P);\FF_p (X\times Y))\]
for a polyhedron $P$ and then building a homology theory to compute $H_q(X\times Y;\FF_p (X\times Y))$.
\end{rem}

\begin{proof}[\textbf{Proof of Lemma~\ref{mlem:push_it}.}] First observe that by Lemma~\ref{lem:framing_half}, the framing can be pushed to the halfspace $H^+$. 
The case of mobile critical points $v$ follows swiftly: By Lemma~\ref{lem:intervals_pq}, the relative $(p,q)$-groups $H_q(\St_v X, \partial \St_v X; \FF_p X)$ vanish,
 provided that $p+q\le n-1$.

It remains only to discuss the case where $v$ is of nontrivial sedentarity, for which we can assume that $X\subset\mathbb{T}^d$. By induction on the sedentarity, we can assume that the claim is proven for sedentarity 
$J\subsetneq I \defeq  \Sed(v)$.

If we consider any face $\tau$ in $\St_v X$; then the inclusion 
\[\wt{\tau} \defeq  \tau\cap (\TR^{[d]{\setminus} I}\times \{-\infty\}^{I})\hookrightarrow \tau\] induces a projection map of local coefficients $(\FF_p X)_{|\tau}\twoheadrightarrow (\FF_p X)_{|\wt{\tau}}\hookrightarrow (\FF_p X)_{|\tau}$ and therefore by linear extension an endomorphism $\mr{p}_I : (\FF_p X)_{|\St_v X}\rightarrow (\FF_p X)_{|\St_v X}$. Using this, we decompose
\[c_{|\St_v X}\ =\ \underbrace{c_{|\St_v X}\ -\ \mr{p}_I (c_{|\St_v X})}_{\defeq \, c_{\mr{sed}}}\ +\ \underbrace{\mr{p}_I (c_{|\St_v X})}_{\defeq \, c_{\mr{mob}}}.
\]
Both chains have the property that their boundaries do not intersect $v$ because $\mr{p}_I$ and $\partial$ commute, and hence they satisfy the assumptions of the lemma. 

\begin{figure}[htb] 
\centering 
 \includegraphics[width=0.98\linewidth]{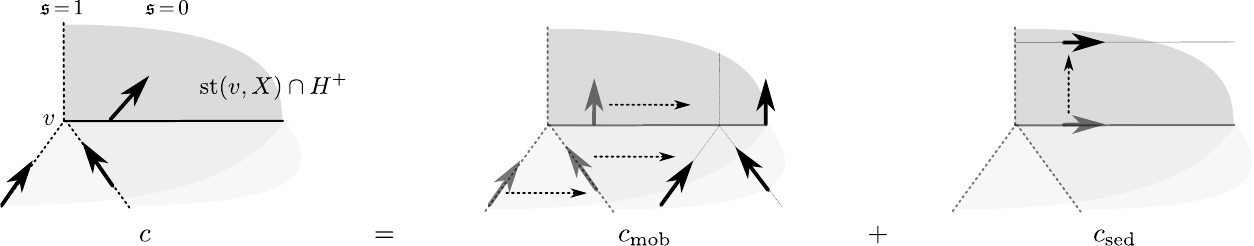} 
 \caption{\small At a face of positive sedentarity, we decompose the chain $c$ into two summands $c_{\mr{sed}}$ and $c_{\mr{mob}}$; the latter is ``mobile'' in the sense that it can be pushed homologously to faces of strictly smaller sedentarity. The same does not apply to $c_{\mr{sed}}$, which can instead be pushed into the interior of $H^+$.} 
 \label{fig:pushing}
\end{figure}

Now, $c_{\mr{sed}}$ can be deformed to a homologous\footnote{Two chains are homologous if their difference is a boundary.} chain within the intersection of $\St_v X$ and the interior of ${H^+}$ in a neighborhood of $v$, using the same arguments as for the sedentarity zero case. 
The chain $c_{\mr{mob}}$ on the other hand can be deformed by homologously pushing it into a coface of $v$ with strictly smaller sedentarity, and then invoking the induction assumption on the sedentarity.
\end{proof}

\begin{proof}[\textbf{Proof of Theorem~\ref{mthm:AKN_trop}}]
As in the proof of Lemma~\ref{mlem:lef_to_convex_cell} (whose notation we adopt for simplicity), we consider
a function $\wt{f}:{P} \longrightarrow \RR^{\ge 0}$ with convex superlevel sets that restricts to a Morse function $f$ on $X\cap P$. Thus, 
we are left to analyse the Morse data at critical loci. Unfortunately, unlike in the mobile situation, we can no longer restrict to the mobile part of a tropical variety as $(p,q)$-groups on $X$ and $X_{|\mo}$ can differ. Hence, we modify our Morse function and define on $P_{|\mo}$
\[\wt{f}(x)\ \defeq\ \prod_{H} \arctan^{\alpha_H}(\mr{d}(H,x))\]
for generic choices of positive reals $\alpha_H$. This can be extended smoothly to the closure $P$ of $P_{|\mo}$, in the sense that the function is smooth when restricted to the subspace of points of a fixed fine sedentarity (intersected with the interior of $P$). Additionally, the function also has strictly convex superlevel sets when restricted to these subspaces, and the level set at $0$ is $\partial P$. The downside is that not all critical points of this function are in $\mathbb{R}^d$.

Geometrically, the Morse data at critical points are but products of Bergman fans and polyhedra of the appropriate dimension. 
Specifically, consider a critical value $t$, some $\varepsilon>0$ small enough, and the inclusion
\[j_{t}:\ X_{\le t-\varepsilon}\ \longhookrightarrow\ X_{\le t},\] where $X_{\le  s}\defeq X\cap f^{-1}(-\infty,s]$ as usual. If the corresponding critical point lies in the relative interior of a face $\sigma$, the Morse data is $(\CO ( \RN^1_\sigma X_{\le t}\ast \partial \sigma), \RN^1_\sigma X_{\le t}\ast \partial \sigma)$.

Hence,
by Lemma~\ref{mlem:push_it} and Lemma~\ref{lem:Kn}, we can deform any $(p,q)$-chain $c$ in $C_q(X_{\le t};\FF_p X)$ with $\supp\, \partial c \in X_{\le t-\varepsilon}$ and $p+q\le n-k-1$ homologously to a chain $\widetilde{c}$ supported in $X_{\le t-\varepsilon}$ with $\partial \widetilde{c}=\partial c$.
In particular, every cycle $z$ in $Z_q(X_{\le t};\FF_p X)$ can be homologously deformed to a cycle $\widetilde{z}$ in $Z_q(X_{\le t-\varepsilon};\FF_p X)$, provided that $p+q\le n-k-1$. Hence, the induced map 
\[j_{t}^\ast:\ H_q(X_{\le t-\varepsilon};\FF_p X)\ \longrightarrow\ H_q(X_{\le t};\FF_p X)\]
is surjective up to dimension $p+q\le n-k-1$. 

Moreover, if $z$ is a boundary, then  $\widetilde{z}$ is the boundary of a chain \[\widetilde{c}\in C_{q+1}(X_{\le t};\FF_p X),\ \ \supp\, \partial \widetilde{c}=\supp\, \widetilde{z} \subset X_{\le t-\varepsilon},\] and we homologously deform $\widetilde{c}$ to be supported in $X_{\le t-\varepsilon}$ if $p+q\le n-k-2$, so that $j_{t}^\ast$ is also injective for these parameters.
Iterated application of this argument at the (discrete) set of critical loci gives the desired result.
\end{proof}

We now turn to the proof of the tropical Kodaira--Spencer Theorem~\ref{mthm:proj_t_hodge_g}.
To start with, we push chains in $X$ to chains in $X\cap H$, as guaranteed by Theorem~\ref{mthm:AKN_trop}, obtaining a chain in $C_q(X\cap H; \FF_p X)$. However, it is not clear in general how to pass from $C_q(X\cap H; \FF_p X)$ to $C_q(X\cap H; \FF_p (X\cap H))$, cf.\ Figure~\ref{fig:ptp}. Hence, Theorem~\ref{mthm:proj_t_hodge_g} is not a direct consequence of Theorem~\ref{mthm:AKN_trop}, and an 
additional argument needs to be made.

\begin{figure}[htb] 
\centering 
 \includegraphics[width=0.8\linewidth]{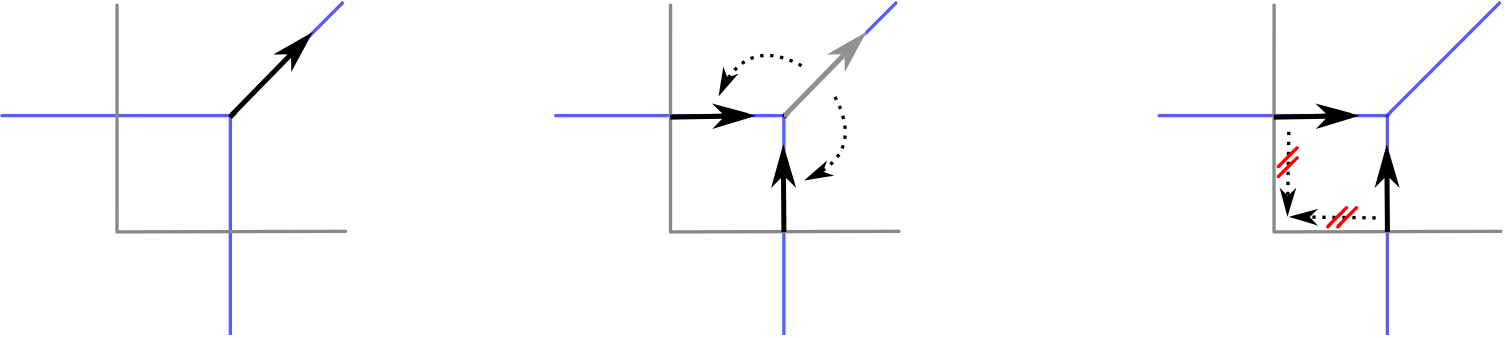} 
 \caption{\small A one-dimensional smooth tropical variety $X$, part of a chamber complex $H$, and a $(1,0)$ chain $c$ in it. While $c$ can be pushed to $X\cap H$, it can not be pushed so that its framing is supported in $\FF_p (X\cap H)$.}
 \label{fig:ptp}
\end{figure}

\begin{proof}[\textbf{Proof of the tropical Kodaira--Spencer Theorem~\ref{mthm:proj_t_hodge_g}}.]
We assume that $p\ge 1$, noting that the case when $p=0$ was dealt with already (since $\FF_0\equiv\mbb{R}$). Moreover, we work with the natural coarsest cell structure on the chamber complex $H$. As the case of stable intersections is more interesting, we focus on that case in detail here, and leave the details of the nonstable case to the reader.

As the case of stable intersections is more interesting, we focus on that case in detail here.

 We divide the proof into two parts by showing that the maps 
\begin{equation*}
H_q(X\cap_{\mathrm{stable}} H;\FF_p X)\longrightarrow H_q(X;\FF_p X) 
\end{equation*}
and 
\begin{equation*}
H_q(X\cap_{\mathrm{stable}} H;\FF_p (X\cap_{\mathrm{stable}} H))\longrightarrow H_q(X\cap_{\mathrm{stable}} H;\FF_p X)
\end{equation*}
induced by inclusion are isomorphisms for $p+q\le n-2$, and onto for $p+q\le n-1$. For this, we establish two claims: 
\begin{compactitem}[$\bullet$]
\item every relative cycle $c\in Z_q(X,X\cap_{\mathrm{stable}} H;\FF_p X)$ is homologous to a chain $\wt{c}\in C_q(X\cap_{\mathrm{stable}} H;\FF_p X)$, and
\item every chain $c\in C_q(X\cap_{\mathrm{stable}} H;\FF_p X)$ with boundary in $Z_{q-1}(X\cap_{\mathrm{stable}} H;\FF_p (X\cap_{\mathrm{stable}} H))$ is homologous  to a chain $\wt{c}\in C_q(X\cap_{\mathrm{stable}} H;\FF_p (X\cap_{\mathrm{stable}} H))$,
\end{compactitem}
as long as $p+q\le n-1$. 

Provided $H$ is in general position and transversal to $X$, the first claim is immediate from Theorem~\ref{mthm:AKN_trop}, since $H$ divides $\mathbb{A}^d$ into pointed polyhedra. For the case of non-general position, we argue by approximation as in Section~\ref{sec:stable}.

Now for the second claim: Every chain $c\in C_q(X\cap_{\mathrm{stable}} H;\FF_p X)$ is a limit of chains $c\in C_q(X\cap (H+tv);\FF_p X)$ for generic $v$ and small enough $t$, so that the homologous deformation carries over provided it works in general position, which we can assume from now. So for the conclusion of the proof, we may 
by Lemma~\ref{lem:push_it} assume that $c$ is a transversal chain.

There are now three situations to consider:
\begin{compactitem}[$\bullet$]
\item If $c\in C_q(X\cap H;\FF_p (X\cap H))$, there is nothing to prove. 
\item Let $\sigma$ denote any facet in the support of $c$, and let $h$ denote the minimal face of $H$ that contains~$\sigma$. By transversality,
we can assume that $\sigma^\circ$ lies in the relative interior of the intersection of $h$ with a facet $S$ of $X$. Then $H' \defeq  H\cap \aff S$ is a chamber complex in $\aff S$. We need a simple lemma.

\begin{lem}\label{lem:M}
Let $\Sigma$ in $ \R^{m}$ be an $(m-1)$-dimensional fan whose complement contains no $j$-dimensional subspace. Then $\FF_p \Sigma= \FF_p  \R^{m}$ for $p\le m-j$. \qed
\end{lem}
 Therefore, if $h$ is of codimension greater than $p$ in $\mathbb{T}^d$, then, by above lemma,
$(\FF_p (H\cap S))_{|\sigma}=(\FF_p X)_{|\sigma}$. The framing of $c$ at $\sigma$ therefore lies in $\FF_p(X\cap H)$.
\item The plan is clear: we need to push $c$ homologously to faces of increasing codimension.
With $h$ a face of $H$ of codimension $\ell\le p$ in $\mathbb{T}^d$, we can write a chain $c_{|h}$ in  $C_q(X\cap h;\FF_{p} X)$ as 
\[c_{|h}\ =\ \underbrace{\sum_{\sigma:\, q\text{-cell}\,\rightarrow\, X} w^\sigma_{p-\ell} \wedge y^\sigma_{\ell}}_{\defeq  c_{|h,\ell}}\ +\ \underbrace{\sum_{\sigma:\, q\text{-cell}\,\rightarrow\, X} w^\sigma_{p}}_{\defeq  c_{|h,0}},\ \] 
where $w^\sigma_\alpha\in (\FF_{\alpha} (X\cap H))_{|\sigma}$ and $y_{\beta}^\sigma \in (\FF_{\beta}X)_{|\sigma}/(\FF_{\beta} (X\cap H))_{|\sigma}$. 
Then, by Theorem~\ref{mthm:AKN_trop}, 
\[
\wt{c}_{|h,\ell}\ \defeq \ \sum_{\sigma:\, q\text{-cell}\,\rightarrow\, X} w_{p-\ell}^\sigma\in C_q(X\cap h;\FF_{p-\ell} (X\cap H))
\]
is homologous to a chain $\wt{c}'_{|h,\ell}$ in $C_q(X\cap \partial h;\FF_{p-\ell}(X\cap H))$, as long as 
\[(p-\ell)+q\ \le \ =\ n-\ell-1.\]
We hence conclude that there exist chains 
${c}'_{|h,\ell}$ and ${c}''_{|h,0}$ in $C_q(X\cap \partial h;\FF_{p} X)$ and $C_q(X\cap h;\FF_{p} (X\cap H))$, respectively, whose sum is homologous to $c_{|h}$. \qedhere 
\end{compactitem}
 \end{proof}

\section{Remarks, examples and open problems.}\label{ssc:rem2}

\subsection{Theorems of Lefschetz type in tropical geometry.} 
A worthwhile long-range goal for this line of research could be 
to understand the Hard Lefschetz Theorem in the tropical setting, and there has been substantial progress on this front in the local case \cite{AHK}. It remains an open problem to formulate and prove the theorems of the Hard Lefschetz and Hodge type for ``K\"ahler'' tropical varieties, or even formalize the notion of K\"ahler structures for tropical varieties. 

\subsection{Lefschetz-type theorems for general tropical varieties}
It is not hard to see that the Lefschetz-type section theorems do not apply to general tropical varieties (which are defined as balanced complexes in a tropical space). In fact,
the Lefschetz property breaks down already for irreducible and balanced $2$-dimensional fans in $\RR^4$, which we will discuss here. Even stronger, such examples exist for $n$-dimensional irreducible and balanced fans in $\RR^d$ for every $d\ge 4$ and $n\le d-2$, for which we refer the reader to \cite{AB18}. This is also sharp, as we shall see below that the Lefschetz theorem applies for any other choice of parameters. 

Indeed,  
as in the construction of the Bergman fan, let $e_1,\cdots,e_5$ denote a generating system of $\mathbb{Z}^{4}\subset\RR^{4}$ such that $\sum_{i=1}^n e_i=0$. If $F$ is any subset of $[5]$, we define
$e_F\defeq \sum_{e_i\in F} e_i$. Consider now the $2$-dimensional fan on rays
\[e_1,\ e_2,\ e_3,\ e_4,\ e_5,\ e_{12}, \ e_{13},\ e_{23},\ e_{145},\ e_{245},\ e_{345},\ \text{and } -e_{45}\]
where we draw a $2$-dimensional cone for every inclusion relation between the atoms (that is, $\{1,2,3,4,5\}$) and subsets of cardinality at least two (that is, $\{12,13,23,145,245,345,45\}$). 

It is easy to see that this fan is balanced and connected. Moreover, if we restrict to the open halfspace with interior normal $-e_{45}=e_{123}$, the result is disconnected, violating the Lefschetz hypothesis. Finally, this disconnectedness is stable under small perturbations of the interior normal, so that even for general position hyperplanes, the Lefschetz hypothesis is false.

So, the intuition that the Lefschetz-type theorem might hold for general balanced complexes is wrong. The source of this confusion is the absence of counterexamples in dimension $\le 3$ and codimension $1$.

\begin{prp}
Let $X$ and $H$ denote two chamber complexes in $\RR^d$.
Then $X$ is, up to homotopy equivalence, obtained from $X\cap H$ by successively attaching $(d-1)$-dimensional cells.

Moreover, let $X$ be a graph in $\R^3$ balanced with positive weights, and
let  $H$ be a chamber complex. Then $X$ is, up to homotopy equivalence, obtained from $X\cap H$ by successively attaching $1$-dimensional cells.\qed
\end{prp}

\begin{proof}
Only the first case is not obvious. However, it can easily be proven using the same Morse-theoretic approach that we used for the Lefschetz theorem for smooth tropical varieties. The crucial ingredient here, the Morse data at critical points, is given by Lemma~\ref{lem:M}.
\end{proof}

\subsection{A stronger version of the Pushing Lemma?}

It is natural to wonder whether it is sufficient
in Lemma~\ref{mlem:push_it} to assume
only that $q\le n-1$ (as opposed to the more restrictive $p+q\le n-1$).
This would reconcile it with the Mikhalkin--Zharkov Transversality Lemma~\ref{lem:push_it}. However, this is not the case, as is shown by the following counterexample.

\begin{figure}[htb] 
\centering 
 \includegraphics[width=0.24\linewidth]{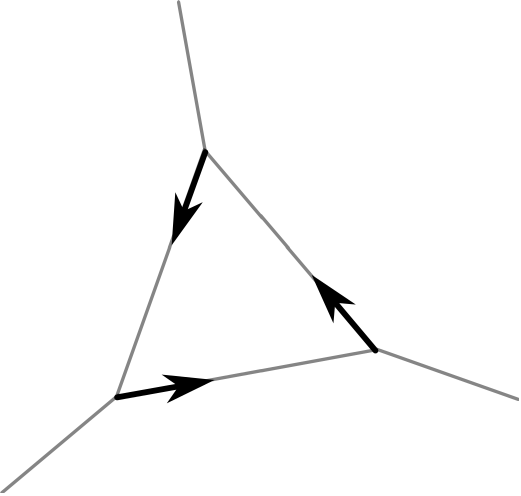} 
 \caption{\small A $(1,1)$-chain that cannot be pushed away from a singularity.}
 \label{fig:nontrivial}
\end{figure}

Consider the uniform matroid $U^3_4$ of rank $3$ on $4$ elements, let $\B$ denote the Bergman fan of $U^3_4$ spanned by a circuit $(e_i)_{i=1,\cdots,4}$ in $\R^3$. 
Choose a closed halfspace $H^+$ that contains $e_1$, $e_2$ and $e_3$ in the interior (and $\mbf{0}\in \partial H^+$), and consider 
the positive span of $e_i$: \ $\rho_i\defeq \pos e_i$ $(i\in \{1,2,3\})$. Let $\{1,2,3\}$ be ordered cyclically, and consider the $(1,1)$-chain
\[c\ \defeq \ \sum_{i=1}^3 ({e_{i-1}-e_{i}})\, \rho_i.\]
Clearly, if we consider the faces $\tau_{i,j}\defeq  \pos \{e_i,e_j\}$, $i\neq j\in \{1,2,3\}$, then a $(1,2)$-chain 
\[\gamma\ \defeq \ \sum_{i=1}^3 (a_i\cdot (e_i-e_{i+1})+b_i\cdot e_{i+1})\, \tau_{i,i+1},\]
in $H^+\cap \B$ with $\supp\ c-\partial \gamma \in \B\setminus\{\mathbf{0}\}$, must satisfy
\[b_i + a_{i+1}+b_{i+1}\, =\, 0\ \ \text{ and }\ \ a_{i+1}-b_{i+1}\, =\, 0\ \ \text{ and }\ \ a_i\, =\, 1\ \ \text{for all $i\in\{1,2,3\}$.}\]
This is inconsistent, cf.\  Figure~\ref{fig:nontrivial}.

\subsection{An integral tropical Kodaira--Spencer Theorem}\label{ssc:intHodge}

The $(p,q)$-homology theory was originally defined by Mikhalkin \cite{Shaw} 
 over the integer lattice.
The impact of this notion on the Lefschetz Theorem is briefly discussed here.

\begin{dfn}[Integral $p$-groups and integral $(p,q)$-groups] Let $\Sigma$ denote any polyhedral fan. For $p\ge 0$, we associate to $\Sigma$ the subgroup $(\FZ_p\Sigma)_{|\mathbf{0}}$ of $\bigwedge^p \mathbb{Z}^d$ generated by elements $v_1\wedge v_2\wedge \cdots \wedge v_p$ where $v_1,v_2,\cdots,v_p$ are integer vectors that lie in a common subspace $\lin{\sigma}, \sigma\in \Sigma$. The groups $(\FZ_p\Sigma)_{|\mathbf{0}}$ are also known as the \emph{integral $p$-groups} of $\Sigma$. 
	
	In analogy with the case of real Hodge groups, the integral $(p,q)$-groups are naturally the homology groups generated by the local system of coefficients $(\FZ_p\Sigma)_{|\mathbf{0}}$.
\end{dfn}

\begin{ex}
	For every simple matroid of rank at least $2$ on ground set $[n]$, we have that
	$(\FZ_1\B)_{|\mathbf{0}}\cong~\Z^{n-1}$.
\end{ex}

\begin{ex}[Orlik--Solomon algebras and integral $p$-groups]
	For every matroid $M$, we have a natural isomorphism between $\mc{OS}^\bullet(M)$, the projectivized Orlik--Solomon algebra of $M$, and $(\FZ^\bullet\B(M))_{|\mathbf{0}}$, the graded algebra of co-$p$-groups of the Bergman fan of $M$~\cite{zbMATH06400709}. 
\end{ex}

It is natural to ask whether the Kodaira--Spencer Lefschetz Section Theorem holds for integral $(p,q)$-groups. Consider this problem:

\begin{quote}
	\emph{Let $X$ denote any $n$-dimensional smooth tropical variety in $\TP^d$, and let $H\subset \TP^d$ denote a generic hyperplane. Is it true that the inclusion $X\cap H\hookrightarrow X$ induces an isomorphism of $(p,q)$-homology up to dimension $p+q\le n-2$, and a surjection in dimension $p+q=n-1$?}
\end{quote}

The answer is clearly negative, essentially because $X\cap H$ can be badly non-smooth. However, even if one gives strong assumptions on the intersection of $X$ and $H$, the problem fails to admit a positive answer, because the integral version of the Pushing Lemma is wrong.

\begin{prp}[Pushing chains fails for integral Hodge theory]\label{prp:pushing_direction}
	There exists a Bergman fan $\B$ of dimension $2$ (pointed at $\mbf{0}$) and a closed general position halfspace $H^+$ in general position w.r.t.\ $\B$ such that
	$(\FZ_1( \B\cap H^+))_{|\mathbf{0}}$ is a strict subgroup of $(\FZ_1 \B)_{|\mathbf{0}}$.
\end{prp}

\begin{proof}
	Let $M$ denote the Fano plane on ground set $\{1,2,3,4,5,6,7\}$ and with rank 
	three flats
	\[\{\{1,2,3\},\{1,4,5\},\{1,6,7\}, \{2,4,6\},\{2,5,7\}, \{3,4,7\},\{3,5,6\}\}.\]
	Then $(\FZ_1 \B(M))_{|\mathbf{0}}= \mathbb{Z}^{6}$.
	Let $\omega=(4,4,4,-3,-3,-3,-3)$, and let 
	$H^+$ be the halfspace determined by the interior normal $\omega$.
	Then \[\mc{L}^{>0} (M)=\{\{1\},\{2\},\{3\},\{1,2,3\}\}.\]

	To compute $(\FZ_1 (\B\cap H^+))_{|\mathbf{0}}$ we now have to compute the integral span of 
	\[\{\sigma \in \B: \sigma \text{ incident to }\mr{Im}\,\mc{L}^{>0}(M)\subset\B \}.\]
	Now, every cone is spanned by the rays of the generating flats, so that
	$(\FZ_1 (\B\cap H^+))_{|\mathbf{0}}$ is generated by the flats
	\[G\ \defeq \ \mc{L}^{>0}(M) \cup \{\{1,4,5\},\{1,6,7\}, \{2,4,6\},\{2,5,7\}, \{3,4,7\},\{3,5,6\}\}.\]
	\begin{figure}[htb] 
		\centering 
		\includegraphics[width=0.23\linewidth]{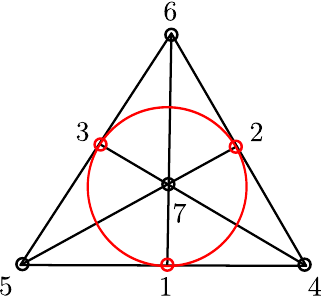} 
		\caption{\small The Pushing Lemma fails for the Fano Plane. Positive flats in red.} 
		\label{fig:Fano}
	\end{figure}
	
	Consider now the weight vector $\vartheta=(0,0,0,1,1,1,1)$. Then 
	\[\vartheta\cdot S\ =\ 0\pmod 2 \]
	for every $S\in G$. But, say, $\vartheta\cdot\{4\}=1$, so $(\FZ_1 (\B\cap H^+))_{|\mathbf{0}}$ is a strict subgroup (of index $2$) in $(\FZ_1 \B)_{|\mathbf{0}}$.
\end{proof}

We are left with the following problem:

\begin{prb}
Is the tropical integral Kodaira--Spencer theorem true for smooth limits of smooth complex projective varieties? 
\end{prb}

{\small 
\renewcommand\refname{\textbf{REFERENCES}}
\bibliographystyle{myamsalpha}
\bibliography{tropicallef}}
\end{document}